\documentclass[12pt,reqno]{amsart}
\usepackage{cite}

\usepackage{verbatim,amscd,amssymb}
\usepackage{graphicx}
\usepackage{amsmath}
	\usepackage{amsmath}
\usepackage{amsfonts}
\usepackage{amssymb}
\usepackage{inputenc}

\usepackage{enumitem}
\usepackage{float}

\usepackage[active]{srcltx}

\graphicspath{ {./images/} }

\usepackage{fontenc}

\usepackage{float}

\newtheorem{theorem}{Theorem}[section]
\newtheorem{lemma}[theorem]{Lemma}
\newtheorem{cor}[theorem]{Corollary}

\theoremstyle{definition}
\newtheorem{definition}[theorem]{Definition}

\theoremstyle{remark}

\theoremstyle{assumption}

\numberwithin{equation}{section}

\newcommand{\uc}{\mathbb{S}}

\newcommand{\RA}{\rightarrow}

\newcommand{\sha}{\succ\mkern-14mu_s\;}

\begin{document}
	
	\date{\today}
	
	\title[Forcing minimal patterns of triods]{Forcing minimal patterns of triods}
	
	\dedicatory{Dedicated to my teacher Professor A.Blokh}
	
	\author{Sourav Bhattacharya}

	\address[Dr. Sourav Bhattacharya]
	{Department of Mathematics, Visvesvaraya National Institute Of Technology Nagpur, 
 Nagpur, Maharashtra 440010, India}
	\email{souravbhattacharya@mth.vnit.ac.in}
	
	\subjclass[2010]{Primary 37E05, 37E15; Secondary 37E45}
	
	\keywords{Sharkovsky Theorem, rotation numbers, triods, twist mappings }

	\begin{abstract}
		
 \emph{Rotation numbers} for some maps of \emph{triods} was introduced in \cite{BMR}. The goal of this paper is to study \emph{patterns} of \emph{triods} which don't force other \emph{patterns}  with the same \emph{rotation number} which we name \emph{triod twists}. We obtain their complete characterization and show that these \emph{patterns} can be conjugated to \emph{circle rotation} by a \emph{piecewise monotone} map. We also describe the dynamics of \emph{unimodal triod twist patterns} with a given rational \emph{rotation number}. 

	\end{abstract}
	
	\maketitle

	\section{Introduction}\label{intro}

An important problem in the theory of \emph{dynamical systems} is to study the co-occurrence of various types of cycles or periodic orbits of a given continuous map. In this connection, the first result is the celebrated \emph{Sharkovsky Theorem}(\cite{shatr})  which furnishes a complete description of all possible sets of \emph{periods} of \emph{cycles} of a given \emph{continuous interval map}. 

\subsection{Sharkovsky's Theorem}

To understand this theorem it is necessary to introduce a special ordering of the set $ \mathbb{N}$  of natural numbers called the \emph{Sharkovsky ordering}: 
	
	$$3\sha 5\sha 7\sha\dots\sha 2\cdot3\sha 2\cdot5\sha 2\cdot7 \sha \dots $$
	$$
	\sha\dots 2^2\cdot3\sha 2^2\cdot5\sha 2^2\cdot7\sha\dots\sha 8\sha
	4\sha 2\sha 1$$ 
	
	If $m\sha n$, then we say that  $m$ is {\it \textbf{sha}rper} than
	$n$. We denote by $Sh(k)$ the set of all natural numbers $m$ such that
	$k\sha m$, together with $k$, and by $Sh(2^\infty)$ the set
	$\{1,2,4,8,\dots\}$ which includes all powers of $2$. For a continuous map $f$, let us denote the set of \emph{periods} of its cycles by $Per(f)$. 
	
	\begin{theorem}[\cite{shatr}]\label{t:shar}
		If $f:[0,1]\to [0,1]$ is a continuous map, $m\sha n$ and $m\in
		Per(f)$, then $n\in Per(f)$. Therefore, there exists $k \in  \mathbb{N}
		\cup \{2^\infty\}$ such that $Per(f)=Sh(k)$. Conversely, if $k\in
		\mathbb{N} \cup \{2^\infty\}$ then there exists a continuous map
		$f:[0,1]\to [0,1]$ such that $Per(f)=Sh(k)$. 		
	\end{theorem}
	
	Theorem \ref{t:shar} elucidates a hidden rich combinatorial structure which controls  the disposition of the cycles of a continuous interval map. However, \emph{period} is a rough characteristic of a cycle as there are a
	lot of cycles with the same period. A much finer way of describing cycles
	is by considering its \emph{cyclic permutation} , that is, the \emph{cyclic
		permutation} we get, when we look at how the map acts on the points
	of the cycle ordered from the left to the right. As it turns out classifying \emph{cycles} in this manner is too detailed and doesn't yield a transparent picture( see \cite{Ba}, \cite{alm00}).  This motivated the development of a  middle-of-the-road way of describing \emph{cycles} : \emph{rotation theory}. 
	
	\subsection{Rotation Numbers}
	The concept of \emph{rotation numbers} was  historically first introduced by Poincar\'e
	for \emph{circle homeomorphisms}(See  \cite{poi}).  It was extended to circle maps of
degree one by Newhouse, Palis and Takens \cite{npt83}, and then
studied, e.g., in \cite{bgmy80}, \cite{ito81}, \cite{cgt84},
\cite{mis82}, \cite{mis89}, \cite{almm88} (see \cite{alm00} with an
extensive list of references). In a general setting \emph{rotation numbers} can be defined as follows: 

\begin{definition}[\cite{mz89},
	\cite{zie95}]
Let $X$ be a \emph{compact metric space} with a Borel $\sigma$-algebra, $\phi:X\to\mathbb R$ be a
\emph{bounded measurable function} (often called an \emph{observable}) and 
$f:X\to X$ be a \emph{continuous} map. Then for any $x \in X$ the set
$I_{f,\phi}(x)$ of all \emph{limits} of the sequence
$ \left \{ {\frac1n} \sum^{n-1}_{i=0}\phi(f^i(x)) \right \}$ is called the {\it
	$\phi$-rotation set} of $x$.  If $I_{f,\phi}(x)=\{\rho_\phi(x)\}$ is a singleton, then the
number $\rho_\phi(x)$ is called the {\it $\phi$-rotation number} of
$x$. 
\end{definition}
 It is easy to see that the $\phi$-\emph{rotation set}, $I_{f,\phi}(x)$ is always a closed
interval. The union of all $\phi$-\emph{rotation sets} of all points of $X$ is called
the \emph{$\phi$-rotation set} of the map $f$ and is denoted by
$I_f(\phi)$.

If $x$ is a $f$-periodic point of period $n$ then its {\it $\phi$-rotation number} 
$\rho_\phi(x)$ is well-defined, and a related concept of the
\emph{$\phi$-rotation pair} of $x$ can be introduced: the pair
$({\frac1n}\sum^{n-1}_{i=0}\phi(f^i(x)), n)$ is the
\emph{$\phi$-rotation pair} of $x$. If the \emph{dynamics} of $f$ and the  \emph{observable} $\phi$ are related we can deduce additional informations about $\phi$-\emph{rotation sets}, for instance, in the case of \emph{circle degree one} case (see \cite{mis82}).

Let
$f:S^1\to S^1$ be a continuous map of degree $1$ and $\pi:\mathbb R\to S^1$
be the natural projection which maps an interval $[0,1)$ onto the whole
circle. Fix a lifting $F$ of $f$. Define $\phi_f:S^1\to \mathbb R$ so that
$\phi_f(x)=F(X)-X$ for any point $X\in \pi^{-1}(x)$; then $\phi_f$ is
well-defined, the \emph{classical rotation set} of a point $z \in S^1$ is simply the $\phi_f$ -\emph{rotation set} 
$I_{f,\phi_f}(z)=I_f(z)$ and the \emph{classical rotation number} of $z$ is is simply the $\phi_f$ -\emph{rotation number} 
$\rho_{f,\phi_f}(z)=\rho(z)$ whenever exists. The {\it rotation set} of the map $f$ is $I_f=\cup I_f(x)$; by \cite{npt83},\cite{ito81} $I_f$ is a closed interval (cf
\cite{B1}). The sum $\sum^{n-1}_{i=0}\phi_f(f^i(x))=m$ taken along the
orbit of an $n$-periodic point $x$ is an integer which defines the {\it rotation pair}
$(m,n)\equiv rp(x)$ of $x$.

\emph{Rotation pairs} can be represented in an interesting way using the notations used in  \cite{BMR}. We  represent the \emph{rotation pair} $rp(x) = (mp, mq)$ with $p,q$ coprime and $m$ being some natural number as a pair $ (t, m)$ where $ t = \frac{p}{q}$. We call the later pair a \emph{modified-rotation pair}(\emph{mrp}) and write $mrp(x) = (t, m)$. We then think of the real line with a prong attached at each rational point and the set $\mathbb{N}
\cup \{2^\infty\} $ marked on this prong in the Sharkovsky ordering $ \sha$ with $1$ closest to the real line and $3$ farthest from it. All points of the real line are marked $0$; at irrational points we can think of degenerate prongs with only $0$ on them. The union of all prongs and the real line is denoted by $\mathbb{M}$. Thus, a \emph{modified rotation pair} $(t,m)$ corresponds to the specific element of $ \mathbb{M}$, namely to the number $m$ on the prong attached at $t$ (See Figure \ref{convex_hull}).  However, no rotation pair corresponds to $(t,2^{\infty})$ or to $(t,0)$. Then, for  $(t_1, m_1) $ and  $  (t_2,m_2)$, in $\mathbb{M}$, the convex hull $[(t_1, m_1), $ $  (t_2,m_2)]$  consists of all \emph{modified rotation pairs} $(t,m)$ with $t$ strictly between $t_1$ and $t_2$ or $t=t_i$ and $m \in Sh(m_i)$ for $ i=1,2$.

\begin{figure}[H]
\caption{Representation of \emph{modified rotation pairs} (\emph{mrp}) using real line and \emph{prongs} attached to it}
\centering
\includegraphics[width=0.5\textwidth]{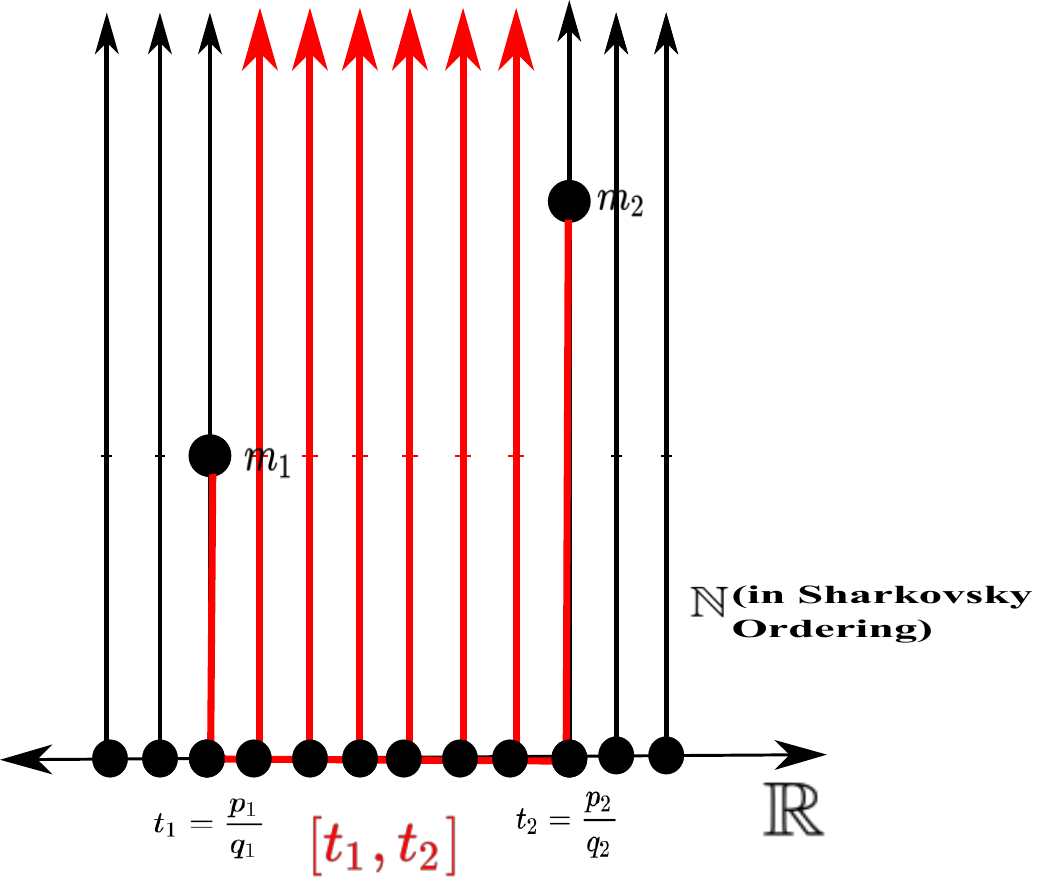}
\label{convex_hull}
\end{figure}

Let $mrp(f)$ be the set of \emph{modified rotation pairs} of all cycles of $f$. See Figure \ref{convex_hull} for a intuitive visualization of $[(t_1,m_1), (t_1,m_2)]$. Clearly, $mrp(f) \subset \mathbb{M}$. The following theorem holds:

\begin{theorem}[\cite{mis82}]\label{circle:maps}
Let $f$ be a degree $1$ map of the circle. Then there are elements $(t_1,m_1)$ and $(t_2,m_2)$ of $\mathbb{M}$ such that $\mathrm{mrp}(f) = [(t_1,m_1), (t_2,m_2)]$ and if $t_i$ is rational, then $m_i \neq 0$ for $i=1,2$. Moreover, for any set of the above form there exists a degree $1$ map $f$ of the circle with $\mathrm{mrp}(f)$ equal to this set.
\end{theorem}

Similar results were also proven for maps of intervals in \cite{BM1} when a version of \emph{rotation numbers} for \emph{interval maps} called  \emph{over-rotation numbers} was introduced using a special observable $\chi$. 

 Let $f:[0,1]\to [0,1]$ be a continuous interval map, $Per(f)$ be its set
of periodic points, and $Fix(f)$ be its set of fixed points. It is easy
to see that if $Per(f)=Fix(f)$ then the \emph{omega limit set}, $\omega(y)$ is a \emph{fixed point} for
any $y$. So, we can assume that $Per(f)\neq Fix(f)$. Now, define:

$$ \displaystyle  \chi_f(x) = \chi(x) =\begin{cases} \frac{1}{2} &\text{if $(f(x)-x)(f^2(x)-f(x))\leqslant  0$,}\\ {0}&\text{if
		$(f(x)-x)(f^2(x)-f(x))>0$.}\end{cases}$$

For any non-fixed periodic point $y$ of period $p(y)$, the integer
$l(y)= \sum^{p(y)-1}_{i=0}\chi(f^i(y))$ is at most $\frac{p(y)}{2}$ and is the
same for all points from the orbit of $y$. The quantity $\frac{l(y)}{p(y)}$ is called the \emph{over-rotation number} of $y$ and the pair $orp(y)=(l(y),
p(y))$ is called the {\it over-rotation pair} of $y$.

	 In an \emph{over-rotation pair} $(p,q)$, $p$ and $q$ are integers and $0<\frac{p}{q}
	\leqslant \frac{1}{2} $.  Like before we can transform all \emph{over-rotation pairs} of \emph{cycles} of a given map $f$ into \emph{modified over-rotation pairs} of $f$ and denote the set of all \emph{modified over-rotation pairs} of \emph{cycles} of $f$ by $mrp(f)$. Then, again $ mrp(f) \subset \mathbb{M}$ and the following theorem holds:

	\begin{theorem}[\cite{BM1}]\label{modified:over:rot}
		If $ f :[0,1] \to [0,1]$ is a continuous interval map with a non-fixed periodic point, then $ mrp(f) = [(t_1, m_1), (\frac{1}{2}, 3) ]$ for some $(t_1, m_1) \in \mathbb{M}$. Moreover, for every $ (t_1, m_1) \in \mathbb{M}$, there exists a continuous map $f : [0,1] \to [0,1]$ with $ mrp(f) = [(t_1, m_1), (\frac{1}{2}, 3)]$. 
	\end{theorem}

	Given an interval map $f$, denote by $I_f$ the \emph{closure of the union of over-rotation numbers} of $f$-\emph{periodic points}, and call $I_f$ the \emph{over-rotation interval} of $f$. By Theorem \ref{t:shar}, any map $f$ with \emph{non-fixed periodic
	points} has a \emph{cycle} of \emph{period} 2 with \emph{over-rotation number} $\frac{1}{2}$;
	by Theorem \ref{modified:over:rot}, if $\rho(P)=\frac{p}{q}$ for a cycle $P$ then
	$[\frac{p}{q}, \frac12]\subset I_f$; hence for any interval map $f$ there exists
	a number $r_f, 0\leqslant r_f < \frac12,$ such that $I_f=[r_f, \frac12]$. 	A \emph{cycle} $P$ is called \emph{over-twist} if it does not \emph{force} any other \emph{cycle} with the same \emph{over-rotation number}. They are the analogs of \emph{twist cycles} for \emph{interval maps}.  By \cite{BM1} and by \emph{properties of forcing relation}, for any
	$\frac{p}{q}\in (r_f, \frac{1}{2})$, $f$ has an \emph{over-twist cycle}
	of \emph{over-rotation number} $\frac{p}{q}$.

	\subsection{Triods}
	
 We now discuss our plans for the paper. An $n$-od is be defined as the set of complex numbers $ z \in \mathbb{C}$ such that $z^n \in [0,1]$. Thus, an $n$-od consists of a \emph{central point} with $n$-copies of the unit interval attached at their end-points.  The form of the set of \emph{periods} of \emph{periodic orbits} of a map of a \emph{triod} (an $n$-od with $n=3$) and that of a general $n$-od with the central fixed was obtained in \cite{alm98} and \cite{Ba} respectively.  A nice description of these results  was discerned in \cite{almnew} where it was shown that the set of periods of such a map can be expressed as the unions of `` initial segments" of the linear orderings associated to all rationals in the interval $(0,1)$ with denominator at most $n$ defined in certain subsets of rational numbers; however this result were only observed and not explained. A complete  explanation of this phenomenon was expounded in \cite{BMR} by obtaining a result similar to Theorem \ref{circle:maps} and Theorem \ref{modified:over:rot} for maps of \emph{triods}. In this paper we continue studying  maps of \emph{triods} in the spirit of \cite{BMR} and produce parallel versions of notions and results earlier known  for  \emph{interval maps} for such maps.

 Let us now discuss our plans in greater details. Let us denote a \emph{triod} by $T$ and its \emph{central point} by $a$. We call each component of $T-\{a\}$, a \emph{branch} of $T$ and  $a$,  the \emph{branching point}. A map $f : T \to T$ is called $P$-\emph{linear} for a cycle $P$ if it fixes $a$, is \emph{affine} on every component of $[P] - (P \cup \{ a\})$ and also constant on every component of $T - [P]$ where $[P]$ is the convex hull of $P$.  We consider the set $\mathcal{U}$ of all continuous maps of $T$ into itself for which the \emph{central point} $a$ of $T$ is the \emph{unique} fixed point. We write $ x>y$ if $x$ and $y$ lie on the same branch of $T$ and $x$ is farther away from $a$ than $y$; write $ x \geqslant y$ if $x>y$ or $x=y$.

We call two cycles $P$ and $Q$ on $T$ \emph{equivalent} if there exists a homeomorphism $h: [P] \to [Q]$ \emph{conjugating} $P$ and $Q$ and \emph{fixing} \emph{branches} of $T$. The class of \emph{equivalence} of a cycle $P$ is called the \emph{pattern} of $P$. A cycle $P$ of a map $f \in \mathcal{U}$ is said to exhibit a \emph{pattern} $A$ or is of \emph{pattern} $A$ or is a \emph{representative} of the \emph{pattern} $A$ in $f$ if $P$ belongs to the equivalence class $A$. A \emph{pattern} $A$ \emph{forces} a \emph{pattern} $B$ if and only if any map $f \in \mathcal{U}$ with a cycle of \emph{pattern} $A$ has also a cycle of \emph{pattern} $B$. It follows (see \cite{alm98}, \cite{alm00}) that if a pattern $A$ forces a pattern $B \neq A$, then $B$ doesn't force $A$.  We say that a cycle $P$ \emph{forces} a cycle $Q$ if the pattern \emph{exhibited} by $P$ \emph{forces} the \emph{pattern} \emph{exhibited}  by $Q$.  We call a cycle and its \emph{pattern} \emph{primitive} if each of its points lies on a different branch of $T$. 

\begin{theorem}[\cite{alm98}, \cite{alm00}]\label{forcing}
	Let $f$ be a $P$-linear map where $P$ is a cycle of pattern $A$. Then a pattern $B$ is \emph{forced} by $A$ if and only if $f$ has a cycle $Q$ of pattern $B$. 
\end{theorem}

Let $ f \in \mathcal{U}$ and  $P \subset T- \{a \}$  be finite. By an \emph{oriented graph} corresponding to $P$ we shall mean a graph $G_P$ whose vertices are elements of $P$ and arrows are defined as follows. For a $x,y \in P$, we will say that there is an \emph{arrow} from $x$ to $y$ and write $x \to y$ if there exists $ z \in T$ such that $ x \geqslant z$ and $ f(z) \geqslant y$. We will refer to a \emph{loop}  in the \emph{oriented graph} $G_P$ as a \emph{point} \emph{loop} to distinguish them from \emph{loops} of \emph{intervals} which we will define next. We call a \emph{point} \emph{loop} \emph{elementary} if it passes through every \emph{vertex} at most once.

If $P$ is a cycle of period $n$, then the loop $\gamma : x \to f(x) \to f^2(x) \to f^3(x) \to \dots f^{n-1}(x) \to x$ , $x \in P$ is called the \emph{fundamental point loop associated with} $P$.  The following result suggests that to find out the patterns forced by a given pattern $A$, it is sufficient to look at the \emph{ point loops} in the \emph{oriented graph} $G_P$ where $P$ \emph{exhibits} $A$. 

\begin{theorem}[\cite{BMR}] \label{loops:orbits:connection:1} The following properties holds:
	
	\begin{enumerate}
		\item Let $x_0 \to x_1 \to \dots x_{m-1} \to x_0$ be a loop in the graph $G_P$. Then, there is a point $y \in X - \{a\}$ such that $f^m(y) = y$ and for every $ k=0,1,2, \dots ,m-1$, the points $x_k$ and $f^k(y)$ lie on the same branch of $X$.

		\item Let $f$ be a $P$-linear map for some cycle $P \neq \{a\}$. Suppose that $y \neq a$ is a periodic point of $f$ of period $q$. Then, there exists a loop $x_0 \to x_1 \to \dots x_{q-1} \to x_0$ such that $ x_i \geqslant f^i(y)$ for all $i$. 
		
	\end{enumerate}

\end{theorem}

We now define \emph{loops} of \emph{intervals}. For this we  borrow the standard definitions from \cite{alm00} and \cite{alm98}. In our \emph{model} of  $T$  we consider $T$ as being \emph{embedded} into the plane with the \emph{central point} at the \emph{origin} and \emph{branches} being segments of straight-lines.

For $x,y \in T$ lying in the same \emph{branch}, we call the \emph{convex hull} $[x,y]$ of $x$ and $y$, an \emph{interval} on $T$ connecting $x$ and $y$. An \emph{interval} $I$ on $T$  is said to $f$-\emph{cover} an \emph{interval} $J$ on $T$ if $f(I) \supset J$. Then, we can speak of a \emph{chain} of \emph{intervals}   $I_0 \RA I_1 \RA \dots $ on $T$ if every previous \emph{interval} on $T$  in the chain $f$-\emph{covers} the next one. We also speak of \emph{loops of intervals} on $T$. Call an \emph{interval} on $T$ \emph{admissible}  if one of its end-points is $a$. We
call a  \emph{chain(a loop)} of \emph{admissible intervals} $I_0,I_1,\dots $ on $T$  an
\emph{admissible loop(chain)} on $T$  respectively. Result similar to Theorem \ref{loops:orbits:connection:1} can also be obtained for \emph{loops of interval} on $T$. The \emph{loop of intervals} $ \Gamma : [x,a]\to [f(x), a] \to [f^2(x), a] \to \dots [f^{n-1}(x), a] \to [x,a], x \in P$ is called the \emph{fundamental admissible loop of intervals} associated with $P$.

\begin{theorem}[ \cite{alm98}, \cite{zie95}]\label{theorem:interval:graph}
	
	For a loop of interval	$ I_0 \to I_1 \to $ $    \dots $ $  I_{q-1}  \to I_0 $ of length $q$ on $T$,  there exists a point $x_0 \in I_0$ satisfying $f^i(x_0) \in I_{i} $ for $ i \in \{0, 1, 2, $ $ \dots q-1\}$ and $f^q(x_0) = x_0$. 
	\end{theorem}

Now, we are in a position to state the \emph{rotation theory} for \emph{triods} as introduced in \cite{BMR}. Let $ f \in \mathcal{U}$ , $P \subset T- \{a\}$ be finite and the \emph{oriented graph} $G_P$ given by $P$ is \emph{transitive}(that is there is a \emph{path} from every \emph{vertex} to every \emph{vertex}). If $P$ is a \emph{cycle}, it is easy to see that $G_P$ is always \emph{transitive}. Call each \emph{component} of $[P] - (P \cup \{ a\})$, a $P$-\emph{basic interval} on $T$ .
We denote the set of all \emph{arrows} of the \emph{oriented graph} $G_P$ by $A$.   

Let us name the branches of $T$ in the anticlockwise direction such that $B = \{ b_i | i = 0,1,2 \}$  (addition in the subscript of $b$  is modulo 3) is the collection of all its branches. Then, we define a \emph{displacement function} $\psi : B \times B \to \mathbb{R}$  by $\psi(b_i, b_{j}) = \frac{k}{3}$ where $ j = i + k $ (modulo 3)



The function $\psi $ induces in a natural way an \emph{displacement function} $ d : A \to \mathbb{R}$ on the set of all arrows $A$ of the oriented graph $G_P$ , namely for $u,v \in T$, defined by $ d(u \to v) = \psi(b_i,b_j)$ where $u \in b_i$ and $v \in b_j$. For a \emph{point loop} $\Gamma$ in $G_P$ denote by $ d(\Gamma)$ the sum of the values of the \emph{displacement} $d$ along the loop. In our model of $T$, this number tells us how many times we \emph{revolved} around the origin in the anticlockwise sense. Thus, $ d(\Gamma)$ is an integer. We call  $ rp(\Gamma) = (d(\Gamma), |\Gamma|)$  and $ \rho(\Gamma) = \frac{d(\Gamma)}{ |\Gamma|}$ as the \emph{rotation pair} and \emph{rotation number} of $\Gamma$  respectively where $|\Gamma|$ denotes length of $\Gamma$.  The closure of the set of \emph{rotation numbers} of all \emph{loops} of $G_P$ is called the \emph{rotation set} of $G_P$,  denoted by $L(G_P)$. By \cite{zie95}, $L(G_P)$ is equal to the smallest interval containing the \emph{rotation numbers} of all \emph{elementary loops} of $G_P$.

The  \emph{rotation number} and \emph{rotation pair} of a cycle $P$ is  the \emph{rotation number} and \emph{rotation pair} of its \emph{fundamental point loop}. The \emph{rotation interval} of  $P$ is  defined to be the  \emph{rotation interval} of $G_P$. We  introduce \emph{modified rotation pairs} for a cycle $P$  in a similar fashion as we did earlier for interval maps and circle maps. Similarly, we speak of the set of all \emph{modified rotation pairs} $mrp(A)$ forced by a given \emph{pattern} $A$. By Theorem \ref{forcing}, $mrp(A)$ is simply equal to the set of all \emph{modified rotation pairs} of \emph{cycles} of the $P$-linear map $f \in \mathcal{U}$ where $P$ exhibits $A$. 
\subsection{Regular Patterns}

\begin{definition}
A \emph{pattern} $A$ for a map $f \in \mathcal{U}$ is called \emph{regular} if $A$ doesn't force a \emph{primitive pattern} of \emph{period} $2$; call a \emph{cycle} $P$ \emph{regular} if it \emph{exhibits} a \emph{regular pattern}. 
\end{definition}

It is easy to see that \emph{patterns} forced by a \emph{regular pattern} are \emph{regular}. A map $f \in \mathcal{U}$ will be called \emph{regular} if all its \emph{cycles} are \emph{regular}. Let $\mathcal{R}$ be the collection of all \emph{regular} maps $ f \in \mathcal{U}$. By Theorem \ref{forcing}, if $P$ is \emph{regular}, then the $P$-\emph{linear} map $f$ is \emph{regular}.

	\begin{theorem}[\cite{BM2}]\label{result:1}

Let $A$ be a regular pattern for a map $f \in \mathcal{U}$. Then there are patterns $B$ and $C$ with modified rotation pairs $(t_1, m_1)$ and $(t_2,m_2)$ respectively such that $mrp(A) = [(t_1, m_1), (t_2, m_2)]$. 
	\end{theorem}

	Theorem \ref{result:1} fully explains the description of the sets of \emph{periods} of \emph{periodic orbits} of maps of \emph{triods} obtained in \cite{alm98} and \cite{almnew}.  Interestingly, Theorem \ref{result:1} is very similar to Theorem \ref{circle:maps} and  \ref{modified:over:rot} established for interval maps and helps us to make conclusions about the dynamics of a \emph{regular map} from reduced information. The next question is: could we explicitly describe the \emph{dynamics} of a \emph{regular pattern}  given its \emph{rotation number} ?  An affirmative answer to this question could be given if we could first describe the simplest  \emph{regular patterns} with a given \emph{rotation number} , that is \emph{regular patterns} which don't  force other \emph{patterns} with the same \emph{rotation number}. We call them \emph{triod twists} keeping analogy with \emph{twist patterns} of \emph{circle maps} and \emph{over twist cycles} of \emph{interval maps}.  
	
		\begin{definition}
		A \emph{regular pattern} $\pi$ is called a \emph{triod twist} if it doesn't force another \emph{pattern} with the same \emph{rotation number}. 
		\end{definition}

	The paper is devoted towards study of \emph{triod twists}. Our main results are :
	
	\begin{enumerate}
		\item We obtain a necessary and sufficient condition for a given \emph{regular pattern} to be \emph{triod twist}. 
		
		\item We investigate the properties of \emph{triod twist} \emph{patterns} and discern a \emph{bifurcation} in the qualitative nature of these \emph{patterns} at \emph{rotation number} $\frac{1}{3}$. We show that a \emph{triod twist pattern} of a given \emph{rotation number} $\rho$ can be conjugated to \emph{circle rotation} by \emph{angle} $\rho$. 
		
		\item  In the end we obtain  a complete description of the dynamics of all possible \emph{unimodal} \emph{triod twist patterns} for a given rational  \emph{rotation number}.
		
	\end{enumerate}

	The paper is divided into three sections: 
	
	\begin{enumerate}[label={(\alph*)}]
		\item Section 1 is \emph{introduction}. 
		
		\item Section 2 recalls {preliminary} facts and ideas needed for our study. 
		
		\item In Section 3  we  obtain a \emph{complete characterization} of a \emph{triod twist pattern} and show that they can be \emph{conjugated} to \emph{circle rotations}. In the end we study \emph{unimodal} \emph{triod twist pattern} with a given rational  \emph{rotation number}.

	\end{enumerate}
\section{Preliminaries}\label{preliminaries}

A map $f \in \mathcal{U} $ is said to  be \emph{monotone} on a set $ U \subset T$ if $ f^{-1} (v)$ is a \emph{connected subset} of $U$ for every $ v \in f(U) $. A set $ U \subset T$ is called a \emph{segment of monotonicity} or a  \emph{lap} of  $f $ if $U$ is the  \emph{maximal}  \emph{open subset} of $T$  on which $f$ is \emph{monotone}(\emph{in terms of set inclusion}). The \emph{number of segments} of \emph{monotonicity} of a \emph{map} $f \in \mathcal{U} $ is called the \emph{modality} of $f$.   The \emph{modality} of a \emph{cycle} $P$ is simply defined as the \emph{modality} of the $P$-\emph{linear} map $f \in \mathcal{U}$. For a \emph{cycle} $P$, call a \emph{map} $ f \in \mathcal{U}$, $P$-\emph{adjusted} if $f$ has no \emph{cycle} other than $P$  with the same \emph{pattern} as $P$.   

\begin{theorem}[\cite{alm98}]\label{P:adjusted}
	For a cycle $P$ of  a map $f \in \mathcal{U}$  , there exists a $P$-adjusted  map $g $ which agrees with  $f$ on $P$, that is, $f |_P = g |_P$. 
\end{theorem}

\begin{figure}[H]
\caption{A \emph{cycle} $P$  with a \emph{green} point $g$, a \emph{black} point $b$ and a \emph{red} point $r$ }
\centering
\includegraphics[width=0.4\textwidth]{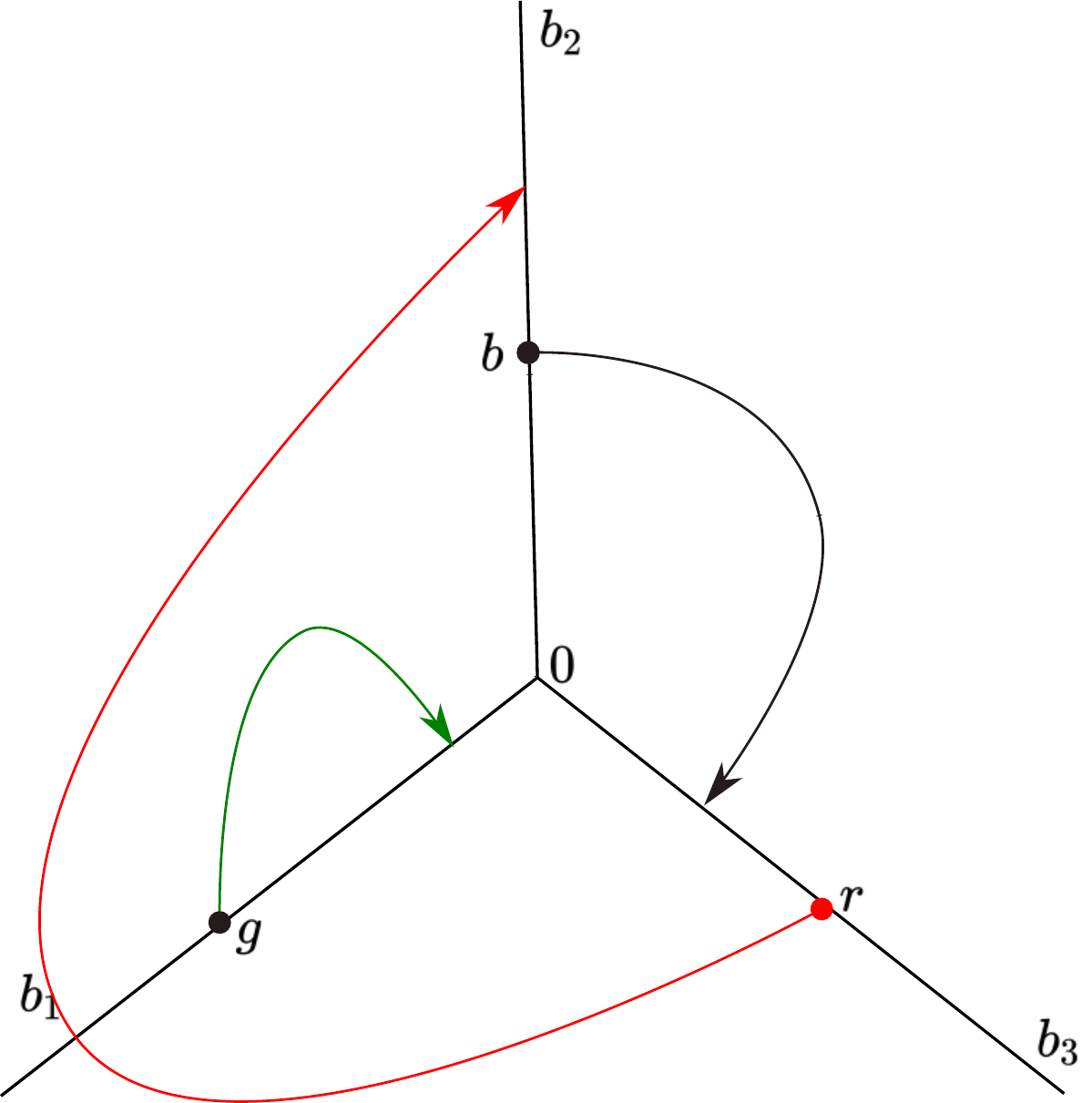}
\label{drawing1}
\end{figure}

While talking about the \emph{arrows} in the \emph{oriented graph} $G_P$ given by a \emph{finite set} $P \subset T - \{ a\}$ we will use the following \emph{color} notations from \cite{BMR}. A arrow $ u \to v$ in $G_P$ where $u,v \in P$ will be called  \emph{green},  \emph{black} or \emph{red} according as $d(u \to v) =  0,  \frac{1}{3}$ or $\frac{2}{3}$ respectively.  If $P$ is a cycle of $f \in \mathcal{U}$, then we define the \emph{color} of a point $x \in P$ to be the \emph{color} of the \emph{arrow} $ x \to f(x)$ in its \emph{fundamental point loop}. A \emph{loop} consisting of \emph{black arrows} will be called a \emph{black loop}.

\begin{theorem}[\cite{BMR}]\label{black:loop:length:3}
	
Let $P$ be a regular cycle. Then for every point $ x \in P$, there exists a black loop of length $3$ passing through it.  
\end{theorem}

\section{Triod twists} \label{phase:function} \label{color:coding}

We say that a \emph{green} point $x$ is mapped \emph{out} if $f(x) > x$ and mapped \emph{in} if  $x > f(x) $

\begin{lemma}\label{df:1:3}\label{no:insider}
	
In a regular cycle $P$ all its green points are mapped in. 
 
\end{lemma}

\begin{proof} 

	Let $f$ be a $P$-linear map. Suppose there exists a \emph{green point} $x \in P$ which is \emph{mapped out}, that is, $f(x) > x$. Since the set of points $\{ z \in P : z > x \}$ is not \emph{invariant} under $f$, there exists a point $ y \in P$, $y >x$ such that  either (i) $ y >  f(y)$ or (ii) $f(y) $ lies in a \emph{branch} different from that of $y$.  In either cases, $f$ has a \emph{fixed point} between $x$ and $y$, a contradiction.

\end{proof}

\begin{lemma}\label{df:1:4} \label{all:branches}
	
A regular cycle $P$ has at least one point in each branch. 
	
\end{lemma}

\begin{proof} 
	
	Suppose $P$ is contained in a single branch $b$ of $T$ and let $x$ and $y$ be the points of $P$ closest and farthest from the branching point $a$. Then, $f(x) \geqslant x$ and $f(y)  \leqslant y$. Thus, $f$ has a fixed point $a' \neq a$ between $x$ and $y$, a contradiction.

If  $P$ is contained in two branches of $T$, we can choose the points $u$ and $v$ of $P$ closest to  $a$ in the two branches. Then, by Lemma \ref{df:1:3},  $f(u) \geqslant v$ and $f(v) \geqslant u$. Thus,  $ u \to v \to u$ is a \emph{loop} in $G_P$ and hence by Lemma \ref{loops:orbits:connection:1}, $P$ forces primitive cycle of period $2$, which is absurd since $P$ is regular.

\end{proof}

\begin{lemma}\label{canonical:numbering}
	
	A regular cycle $P$ always forces a primitive cycle of period $3$.

\end{lemma}

\begin{proof}

 By Lemma \ref{all:branches}, each loop in $G_P$ must pass through each branch at least once. Let $\gamma$ be the loop of shortest length in $G_P$. Suppose $\gamma$ has two arrow $A_1$ and $A_2$ ending in the same branch at $x_1$ and $x_2$ respectively where $x_1 > x_2$. Then, we replace $A_1$ with an arrow $A_3$ beginning where $A_1$ begins and ending at $x_2$. With this replacement we get a shorter loop. Continuing in this manner after finitely many steps, we get a loop $\gamma'$ which doesn't have more than one arrow ending in the same branch. By Lemma \ref{loops:orbits:connection:1}, $\gamma'$ clearly forces a primitive cycle of period 3.

\end{proof}

	Let the pattern associated with the primitive cycle of period $3$ be denoted by $\pi_3$. It is easy to see that $\pi_3$ has rotation number $\frac{1}{3}$ with respect to a numbering $\{ b_i | i = 0,1,2 \}$  (addition in the subscript of $b$  is modulo 3) of the branches of $T$. We now start investigating the properties of a triod twist pattern.

\begin{lemma}\label{coprime:rotation:pair}
	If a triod twist pattern $\pi$ has rotation pair $(p,q)$ then $p$ and $q$ are not coprime. 
\end{lemma}

\begin{proof}
	Let $P$ be a cycle which exhibits $\pi$ and suppose $p$ and $q$ are not coprime. Then, there exists coprime integers $s$ and $t$ such that  $\frac{p}{q} = \frac{s}{t}$ and   $ u = \frac{p}{s}= \frac{q}{t} > 1$.  The \emph{modified rotation pair} of $P$ is $(\frac{s}{t}, u)$. By Theorem \ref{result:1}, $P$ must force a cycle with \emph{modified rotation pair} $(\frac{s}{t}, 1)$ and hence a cycle with \emph{rotation pair} $(s,t)$. Thus, $P$ forces another cycle $Q \neq P$ with the same rotation number, a contradiction. 
	
\end{proof}

 For an ordering  $\{ b_i | i = 0,1,2 \}$  (addition in the subscript of $b$  is modulo 3) of the branches of the triod and for a regular cycle $P$ we denote by $p_i$, the point of  $P$ closest to the branching point $a$ in each branch $b_i$. Then the following result follows :

\begin{lemma}\label{all:black:all:green}
Let $P$ be any regular cycle. Then, the points $p_i, i = 0,1,2 $ are either all black or all	red. 
\end{lemma}

\begin{proof}
Suppose $p_0$ is black. Then, $f(p_0)  \geqslant p_1 $. If $p_1$ is \emph{red}, then $f(p_1) \in b_0$ and hence $f(p_1) > p_0$. Thus, $p_0 \to p_1 \to p_0$ is a loop in $G_P$ and hence by Lemma \ref{loops:orbits:connection:1}, $P$ forces a primitive cycle of period $2$ a contradiction. Thus, $p_1$ is black. Similarly, it follows that $p_2$ is black. If $p_0$ is red , then using similar reasoning it follows that $p_1$ and $p_2$ are also red. 
	
\end{proof}

Clearly, if  $p_i$ are all \emph{red} for all $ i \in \{ 0,1,2\}$, we can re-number the \emph{branches} in such a manner that  all $p_i$ are \emph{black}. Hence, we have :

\begin{cor}\label{canonical:ordering}
	For any regular cycle $P$, there exists a  ordering  $\{ b_i | i = 0,1,2 \}$  (addition in the subscript of $b$  is modulo 3) of the branches of the triod such that $p_i, i = 0,1,2 $ are all black. We call this ordering the canonical ordering. 
\end{cor}

From now on throughout the rest of the paper, we will assume that the branches has been \emph{canonically ordered}.  Lemma \ref{coprime:rotation:pair}  implies the following Lemma :

\begin{lemma}\label{rot:one:third}
 $\pi_3$ is the unique triod twist pattern with rotation number $\frac{1}{3}$. 
\end{lemma}

\begin{proof}
	Let $P$ be a triod twist cycle with rotation number $\frac{1}{3}$ and let $m_1, m_2, m_3$ be the number of its \emph{green}, \emph{black} and \emph{red} points respectively. Since, $P$ has \emph{rotation number} $\frac{1}{3}$, simple computation yields $m_1 = m_3$, that is, $P$ must have \emph{equal number} of \emph{green} and \emph{red} points. By Lemma \ref{coprime:rotation:pair}, the period of $P$ must be $3$. Thus, $m_1 < 2$. Thus, from Lemma \ref{all:black:all:green} it follows that $m_1 = m_3 =0$ and $m_2 = 3$ and hence the result follows.

\end{proof}

\begin{figure}[H]
\caption{The \emph{triod twist pattern} $\pi_3$ with rotation number $\frac{1}{3}$}
\centering
\includegraphics[width=0.4\textwidth]{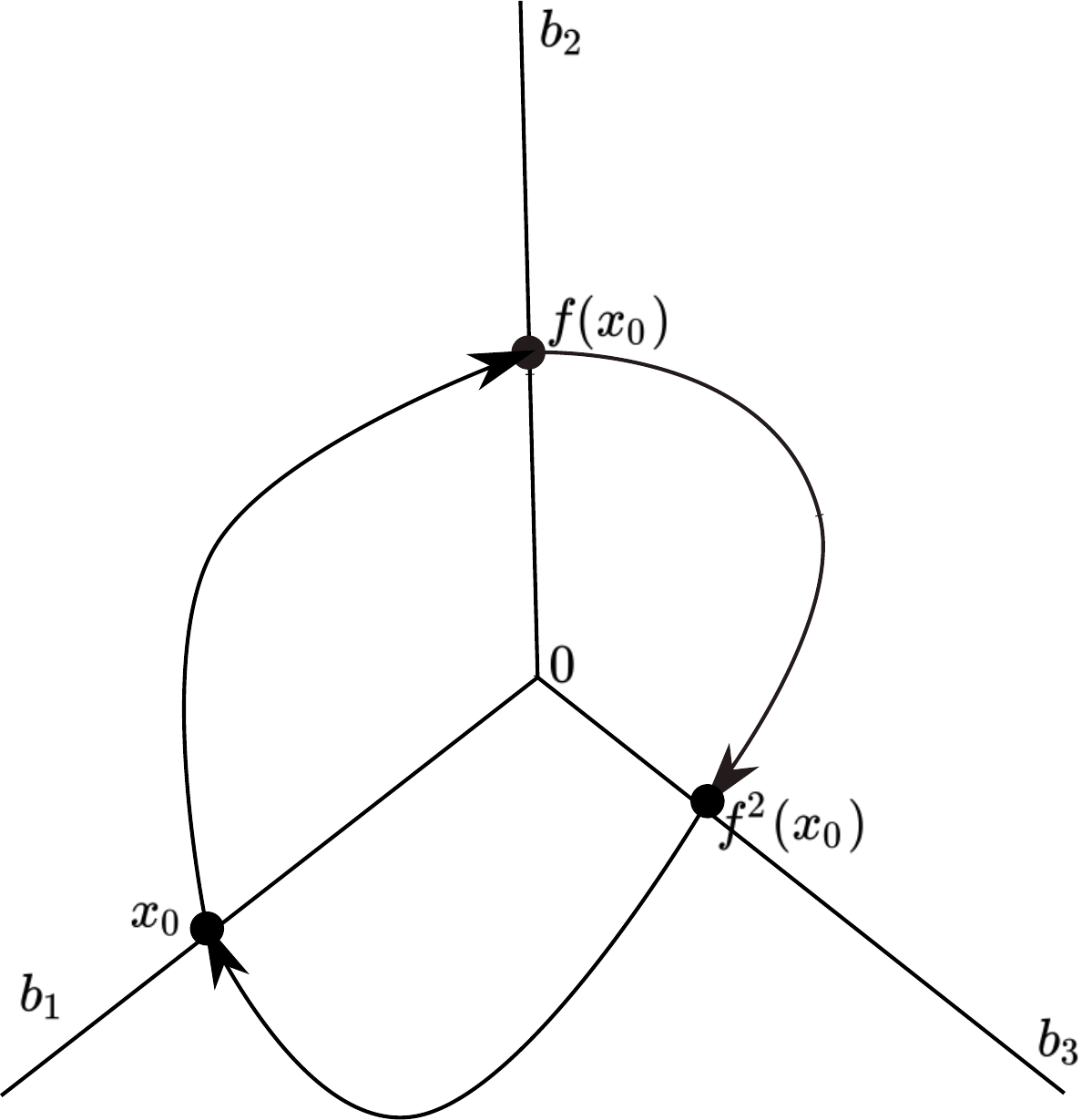}
\label{drawing3}
\end{figure}

We now study \emph{triod twist patterns} of rotation number $\rho \neq \frac{1}{3}$.

\begin{lemma}\label{no:red}
	A triod twist  pattern $\pi$  with rotation number $\rho <  \frac{1}{3}$ has no red points. 
\end{lemma}

\begin{proof}
	Let $P$ be a \emph{cycle} which exhibits $\pi$ and let $f$  be  a $P$-adjusted map. Suppose $P$ has a \emph{red point} $x$. By Theorem \ref{black:loop:length:3} , there exists a \emph{black $f$-point loop} $  x \to x_1 \to x_2 \to x$ of \emph{length} 3 passing through $x$. It follows that $y=f(x)$ and $x_2$  lie in the same branch of $T$ and hence either $y \geqslant x_2$ or $x_2 >y$.  If $ y \geqslant x_2$ then we get a $f$-\emph{point loop} $x \to x_2 \to x$ and hence by Theorem \ref{loops:orbits:connection:1}, $P$ forces a primitive cycle of period $2$, a contradiction. 
	
	Now, suppose $x_2 > y$. Let $\gamma $  be the \emph{fundamental point loop}  associated with the periodic orbit $P$.  Since, $ x_2 > y$, we can replace  $ x \to y$ in  $\gamma$ by  $ x\to x_1 \to y$ to obtain a new \emph{point loop} $\gamma'$. Clearly, $\gamma'$ has \emph{greater length} and hence \emph{lower rotation number} than $\rho$. Hence, by Theorem \ref{loops:orbits:connection:1}, $P$ forces a cycle $R$ with rotation number than $\rho'$ such that $\rho' < \rho < \frac{1}{3}$. Then, Theorem \ref{result:1} and Lemma \ref{canonical:numbering} guarantees the existence of cycle $Q \neq P$ of $f$ with rotation number $\rho$, a contradiction since $P$ is a \emph{triod twist cycle}. 
	
\end{proof}

\begin{lemma}\label{no:green}
	A triod twist  pattern $\pi$  with rotation number $\rho >  \frac{1}{3}$ has no green points. 
\end{lemma}

\begin{proof}
	Let $P$ be a \emph{cycle}  which exhibits $\pi$ and let $f$ be a $P$-\emph{adjusted} map.  Suppose $P$ has a \emph{green point} $x$. By Lemma \ref{no:insider} , $x > f(x)$.  So, we can delete the \emph{arrow} $ x \to f(x)$ from the fundamental point loop $\delta$ associated with $P$ to get a new \emph{point loop} $\delta'$ having greater rotation number than $\rho$. Thus, by Theorem \ref{loops:orbits:connection:1},  $f$ has a cycle  $U$ with \emph{rotation number}  strictly greater than $\rho$.   Hence, by Theorem \ref{result:1} and Lemma \ref{canonical:numbering}, $f$ has a cycle $S \neq P$ with \emph{rotation number} $\rho$ , a contradiction !
\end{proof}

Lemma \ref{no:red}, Lemma \ref{no:green} and Lemma \ref{rot:one:third} can be represented pictorially using a bifurcation diagram shown in Figure \ref{drawing2}.

\begin{figure}[H]
\caption{Bifurcation diagram showing the change in the color of points with the change in rotation number $\rho$}

\includegraphics[width=0.7\textwidth]{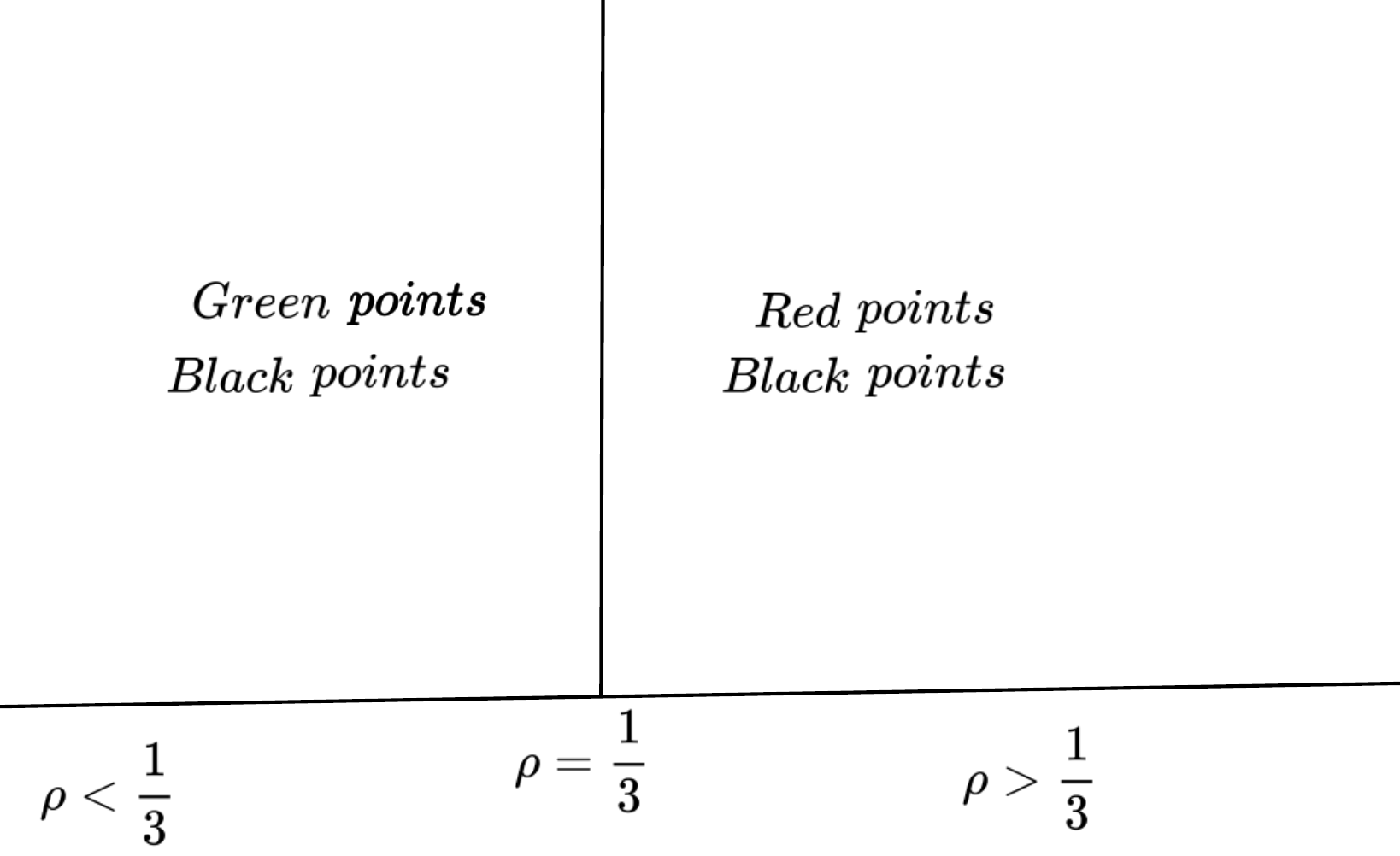}

\label{drawing2}
\end{figure}

We now obtain a complete characterization of \emph{triod twist patterns}. For this we borrow a few notions introduced by A.Blokh and M. Misiurewicz in \cite{BM2} for interval maps. We will use them in a slightly modified way for \emph{triod} maps.

\begin{definition}
	
A \emph{regular} cycle $P$ is called \emph{green} if for any $x,y \in P$ , $ x > y $ such that $f(x)$ and $f(y)$ \emph{lies in the same branch of} $T$, we have $ f(x) >  f(y)$. Call a \emph{pattern} $\pi$ \emph{green} if any \emph{cycle} which exhibits it is \emph{green}.

\end{definition}

\begin{definition}
	The \emph{code function} for a \emph{cycle} $P$  of a map $ f \in \mathcal{R}$ with \emph{rotation number} $\rho \neq \frac{1}{3} $   is a function $ \psi: P \to \mathbb{R}$    defined as follows: 
	
	\begin{enumerate}
		\item choose any point $x_0 \in P$ and set $ \psi(x_0) =0$
		
		\item define $\psi$ iteratively on other points of $P$  as follows :  
		
		\raggedright for each $ k \in \mathbb{N} $ set $ \psi \left ( f^k(x_0) \right ) = \psi(x_0) + k \rho - [t_k]$ where  $ t_k = \displaystyle \sum_{j=0}^{k-1} d \left (f^j(x_0), f^{j+1}(x_0) \right )$ and for $x \in \mathbb{R}$, $[x]$ denotes the greatest integer less than or equal to $x$.

	\end{enumerate}
	
	For any $ x \in P$ , the value $\psi(x)$ is called the \emph{code} of the point $x$. 
\end{definition}

If $P$ is periodic of period $q$,  then $t_q = q \rho \in  \mathbb Z_+$ and hence $ \psi(f^q(x_0)) = \psi(x_0)$. So the definition is consistent.  It is easy to see that  $ \psi(x)- \psi(y)$ is \emph{independent of the choice of the point} $x_0$ for $x,y \in P$.

\begin{definition}

The \emph{code function} $\psi$  of a \emph{cycle} $P$ with \emph{rotation number} $\rho \neq \frac{1}{3}$  is said to be \emph{non-decreasing} if the following holds: 

\begin{enumerate}
	\item If $ \rho < \frac{1}{3}$, then  $\psi(x) \leqslant \psi(y )$ whenever $ x> y$.
	
	\item If $ \rho > \frac{1}{3}$, then $\psi(x) \geqslant \psi(y )$ whenever $ x> y$.
\end{enumerate}

The \emph{phase function} $\psi$  of a \emph{cycle} $P$  is called  \emph{strictly increasing} if $\psi$ is \emph{non-decreasing} and the \emph{phases} of no two \emph{consecutive points} of $P$ lying in the \emph{same branch} are the same.

\end{definition}

\begin{lemma}\label{tri-od:rot:twist:order:inv}
	A triod twist  pattern $\pi$  is green. 
\end{lemma}

\begin{proof}
	Let $\pi$ be a \emph{triod twist  pattern} of period $n$. By Theorem \ref{P:adjusted},  there exists a $P$-\emph{adjusted} map $f$ for a cycle $P$ exhibiting $\pi$.  Suppose that there are points $y,z \in P$ such that $ y >  z$ and $f(z) > f(y)$.  Let $\Gamma $ be the  \emph{fundamental admissible loop of intervals associated with}  $P$. If $f(x) = y$ for $x \in P$,  we can modify $\Gamma$ by replacing $[x, a] \to [y, a] \to [f(y), a ]$  by $[x,a] \to [z,a] \to [f(y), a]$ to get a new \emph{loop of intervals}  $\Delta$. By construction the cycle $Q$ of $f$ \emph{associated} with $\Delta$ has the same \emph{rotation number} as $P$.

 Suppose if possible $P =Q$. Choose a point $ z  $ of $P$ which is  \emph{closest} to  $a$ on its \emph{branch}. Then the interval $[z,a]$  appears in both the \emph{loops} $\Gamma$ and $\Delta$. Choose $ k \in \mathbb{N}$ such that $ y = f^k(z)$. Since, $z$ is the unique point of $P$ in the interval $[z,a]$ and $P$ is associated with $\Delta$,  so starting at $ [z,a]$ and following the \emph{trajectory} of the point $z$ along the loop $\Delta$ for $k$ steps, we should land at $ y \in [z,a]$ which is a contradiction since $ y > z$. Hence, $ P$ is \emph{not associated} with $\Delta$ and hence $ Q \neq P$. 
	
	So, $f$ has a cycle $Q \neq P$ with the  \emph{same rotation number} as $P$. Since, $f$ is $P$-\emph{adjusted}, $P$ and $Q$ have \emph{different  patterns}. So, $\pi$ forces another pattern $\pi'$ with the same \emph{rotation number}, a contradiction !

\end{proof}	

\begin{lemma}\label{necessary:condition:tri-od:rot:twist}
	A triod twist  pattern $\pi$ with rotation number $ \rho \neq \frac{1}{3}$ has a strictly increasing phase function. 
\end{lemma}

\begin{proof}
	
	Consider a cycle $P$  exhibiting  $\pi$ and  let $f$ be a $P$-\emph{adjusted map}. First consider the case $ \rho < \frac{1}{3}$. Suppose that there are points $x,y \in P$ such that $ x>y$ and $ \psi(x) \geqslant \psi(y)$.  Choose $ k \in \mathbb{N}, k < q$ with $ x = f^k(y)$. Consider the \emph{sequence of intervals} $L_i, i \in \{0, 1,2, \dots k\}$ such that $L_0 = [y,a]$, $L_1 = [f(y), a]$, $L_2 = [f^2(y), a] , \dots, L_{k-1} = [ f^{k-1}(y), a], , L_{k} = [ f^{k}(y), a]= [x,a] \supset [y,a]$.  Since, $ x > y$ ,  $L_0 \to L_1 \to L_2 \to \dots L_{k-1} \to L_0$ is an \emph{admissible loop} of intervals. By Theorem \ref{theorem:interval:graph}, $f$ has a point $z \in L_0$ such that $f^i(z) \in L_i, i \in \{ 0,1,2 \dots k-1\}$ and $f^k(z) = z$. The \emph{period} of the orbit $Q$ of  $z$ is either $k$ or a divisor of $k$ and hence the \emph{rotation number} of $Q$ is  $  \frac{\ell}{k}$ for some integer $\ell $. The \emph{phase} of the point $x$ is $\psi(x) = \psi (f^k(y)) = \psi(y) + k \rho -  \sum_{i=0}^{k-1} d(f^{i}(y), f^{i+1}(y)) =   \psi(y) + k \rho -  \ell $. Now,  $ \psi(x) \geqslant \psi(y)$  yields $ \frac{1}{3} > \rho \geqslant  \frac{\ell}{k}$. Thus, by Theorem \ref{result:1},  $f$ must have a  cycle $S \neq P$ of \emph{rotation number} $\rho$ , a contradiction !  The case when $\rho >  \frac{1}{3}$ follows similarly taking into consideration the other side of $\frac{1}{3}$.

\end{proof}

\begin{lemma}\label{sufficient:condition:tri-od:rot:twist}
An regular pattern $\pi$ which is green and has a strictly inc	reasing phase function is a triod twist.

\end{lemma}

\begin{proof}
	
Let $P$ be a \emph{regular cycle} exhibiting $\pi$  with \emph{rotation number} $\rho$ and let $f$ be a $P$-\emph{linear map}. We first consider the case $ \rho < \frac{1}{3}$. Suppose $f$ has a cycle $R \neq P$ of \emph{rotation number} $\rho$.  Define a function $h : R \to P$ as follows. Take $t \in R$. Since, $f$ is $P$-\emph{linear},  $t$ lies in some $P$-\emph{basic interval} $[x,y] $ $(x> y)$. If $f(x)$ and $f(y)$ lie in the same \emph{branch}, define $h(t) = x$. Otherwise since $f$ is $P$-\emph{linear},  $f(t)$ must lie in the same \emph{branch} as $f(x)$ or with $f(y)$. In the first case, define $h(t) = x$, else define, $h(t) =y$. Since $P$ is \emph{green},  $h(f(t)) \geqslant h(f(t))$.  So, for every $ t \in R$, $ \Gamma_t :  [ h(t), $ $a] \to $ $  [h(f(t)) ,  $ $ a] \to  $ $  [h(f^2(t)) , a] \to  $ $ \dots [h(t), a] $ is an \emph{admissible loop of intervals} having  \emph{rotation number} $\rho$.

Let $u_0^t = h(t), u_1^t = h(f(t)), u_2^t = h(f^2(t)), \dots$ and so on. Thus, $\Gamma_t : [u_0^t, a] \to [u_1^t,a] \to \dots [u_0^t, a]$ where $u_i^t \in P$ and $f(u_i^t) > u_{i+1}^t$. If $u_{i+1}^t \neq f(u_i^t)$ for some $i$, then $u_{i+1}^t = f^j(u_i^t)$ for some $j >1$ and so we can replace $[u_i^t, a] \to [u_{i+1}^t, a]$ with  $[u_i^t, a] \to [f(u_i^t), a] \to [f^2(u_i^t), a] \to \dots [f^{j-1}(u_i^t), a] \to [u_{i+1}, a]$. Doing this for every $i$ with $ u_{i+1}^t \neq  u_i^t$ we get a new admissible loop $\Delta_t$ which is either the \emph{fundamental admissible loop of intervals of  $P$ or its repetition} and so has \emph{rotation number}   $\rho$. By construction, the \emph{rotation number} of $\Delta_t$ is the weighted average of the \emph{rotation number} of $\Gamma_t$ and \emph{inserts}(here by \emph{rotation number} of an \emph{insert} $ [f(u_i^t), a] \to [f^2(u_i^t), a] \to \dots [f^{j-1}(u_i^t), a] $ we simply mean the quantity $\displaystyle $ $  \frac{1}{j}  $ $  [ (d(u_i^t, f(u_i^t)) + \dots + $ $ d(f^{j-1}(u_i^t),u_{i+1}^t ) ] $).

Consider an \emph{insert} $[u_i^t, a] $ $  \to [f(u_i^t), $ $   a]  $ $ \to  $ $   \dots [f^{j-1}(u_i^t), $ $  a]    \to [u_{i+1}^t, a]$  and let its \emph{rotation number} is $t$. Simple computation shows that since $u_i^t > u_{i+1}^t$ we have $ \psi(u_i^t) < \psi(u_{i+1}^t)$ and thus, $t < \rho$. Thus, the \emph{rotation number} of every \emph{insert} is \emph{strictly less than} $\rho$. The \emph{rotation number} $\rho$ of $\Delta_t$ is the weighted average of the \emph{rotation number} $\rho$ of $\Gamma_t$ and the \emph{rotation numbers} of \emph{inserts} each of which is \emph{strictly less than} $\rho$, a contradiction.

In the case $ \rho >  \frac{1}{3}$,  the  \emph{rotation number} of each \emph{insert} will be \emph{strictly greater than} $\rho$ and the proof follows mutatis mutandis.

\end{proof}	

Combining Lemma \ref{tri-od:rot:twist:order:inv}, Lemma \ref{necessary:condition:tri-od:rot:twist} and Lemma \ref{sufficient:condition:tri-od:rot:twist} we get : 

\begin{theorem}\label{characterization:tri-od:rot:twist}
The necessary and sufficient condition for a regular pattern $\pi$  to be a  triod twist  is that it is  green and has a strictly increasing code function. 
\end{theorem}

Let $P$ be a \emph{triod twist cycle} of \emph{rotation number} $\frac{p}{q}$ and let $f$ be a $P$-linear map. We define a \emph{ relation} $\sim$ on $P$ as follows. For $x,y \in P$, $ x \sim y$ if $x,y$ lie in the same \emph{branch} of $T$ and  points of $P$ lying in $[x,y]$ have the same \emph{integral part} of the \emph{code function} $\psi$. It follows easily that $\sim$ is an \emph{equivalence relation } on $P$ and the resultant \emph{ equivalence classes}  are called \emph{concordant pieces} of $P$. Then, the following holds :

\begin{theorem}
	
	There exists a map $\Theta : T \to T$ which conjugates $P$ with circle rotation on $\uc$ by angle $\frac{p}{q}$. Further, $\Theta$ is   monotone on the convex hulls of concordant pieces of $P$. 
	
\end{theorem}

\begin{proof}
	
	 We parametrize  $\uc$ by $ [0,1]$ with its endpoints identified and define $ g : [0,1) \to [0,1) $ by $g(x) = x + \frac{p}{q}$ (mod 1).  Let $Q$ be the orbit of $0$ for this map. We define  $\Theta : P \to Q$ as follows. Pick any arbitrary point $x_0 \in P$ and set $ \Theta(x_0) = \psi(x_0) = 0$ and  $ \Theta (f^i(x_0)) =  \psi (f^i(x_0)) $ (mod 1) ,  $i \in  \mathbb{N}$. We have, for each $ x\in P$,  $ \Theta(f(x)) = \psi(f(x)) $(mod 1) $ = \psi(x) + \displaystyle \frac{p}{q} $ (mod 1)  $ = g( \psi(x))$. We extend $\Theta$ from $P$ to $T$ by defining it \emph{linearly} on each $P$-\emph{basic interval} and such that it is \emph{constant} on every \emph{component} of $T - [P]$ where $[P]$ is the \emph{convex hull} of $P$.  The rest follows from Theorem \ref{characterization:tri-od:rot:twist}.

\end{proof}	

We  will now exclusively describe \emph{dynamics} of all possible \emph{unimodal triod twist patterns} with a given \emph{rational rotation number}.  We begin by introducing a few notions. Sets of consecutive points of a certain \emph{color} lying in a \emph{branch} is said to form a \emph{block} of that \emph{color}, thus, we have \emph{green blocks}, \emph{black blocks} and \emph{red blocks}. For $A,B \subset T$, we write $A>B$ if $ a >b$ for every $a \in A$ and $b \in B$.

By Lemma \ref{rot:one:third}, $\pi_3$ is the unique \emph{triod twist pattern} with \emph{rotation number} $\frac{1}{3}$.   Let $P$ be a \emph{periodic orbit} which \emph{exhibits} a \emph{unimodal triod twist pattern} with \emph{rotation number} $\rho = \frac{p}{q} \neq \frac{1}{3}$, $g.c.d(p,q) =1$ and let $f$ be a $P$-\emph{linear map}. Let us assume that the \emph{branches} $b_i , i =0,1,2$ are \emph{canonically ordered} (See Lemma \ref{all:black:all:green} and Corollary \ref{canonical:ordering} ) such that the \emph{color block} $B_i$ closest to the \emph{branching point} $a$ in each \emph{branch} $b_i, i =0,1$ is \emph{black}. Now, since $P$ is \emph{unimodal}, $P$ can have at most one \emph{block} $C$ of a different \emph{color} such that  $ C > B_j$ for some $ j \in \{1,2, 3\}$.  The remaining branches $b_{j+1}$ and $b_{j+2}$ (addition in the subscript $j$ is modulo $3$) will have only one \emph{block} each namely $B_{j+1}$ and $B_{j+2}$.

We first consider the case when $\frac{p}{q} < \frac{1}{3}$. By Lemma \ref{no:red},  $C$ is a \emph{green}. Let $ |C| = g$ and $ |B_j | = b$. By Lemma \ref{tri-od:rot:twist:order:inv}, the points of the set $B_{j}$ map onto the set $B_{j+1}$ in an order-preserving manner, hence, $ |B_{j+1} | = b$. Similarly, $|B_{j+2}| = b$. Since, the \emph{rotation number} of $P$ is $\frac{p}{q}$, simple computation yields $ b = p$ and $ g = q -3p$. Thus, the branches $b_j$, $b_{j+1}$ and $b_{j+2}$ (addition in the subscript $j$ is modulo $3$) contains $ q-2p, p$ and $p$ points respectively. Let us name them as $ x_1, x_2, \dots x_{q-2p}$ ,  $ y_1, y_2, \dots y_{p}$ and $z_1, z_2, \dots z_p$ respectively ( in the direction away from  $a$). 

\begin{enumerate}
	\item By  Lemma \ref{tri-od:rot:twist:order:inv}, the first $p$ points of the \emph{branch} $b_j$ (in the direction away from  $a$) map to the $ p$ points of the \emph{branch} $b_{j+1}$ in an \emph{order-preserving fashion}  and these $p$  points then further map to the $p$ points of the \emph{branch} $b_{j+2}$  in an \emph{order-preserving fashion}. In other words,  $f(x_i) = y_i$ and $f(y_i) = z_i$ for $ i = 1,2 \dots p$.
	
	\item By Lemma \ref{no:insider}, the last point $x_{q-2p}$ in the branch $b_j$ (in the direction away from $a$) must be the image of a point in the \emph{branch} $b_{j+2}$ (addition in the subscript $j$ is modulo $3$). By Lemma \ref{tri-od:rot:twist:order:inv}, it follow that $f(z_p) = x_{q-2p}$. Now, since $P$ is \emph{unimodal} it follow that the images of the points in the \emph{branch} $b_{j+2}$ cannot expand when the map to the \emph{branch} $b_j$. Hence, by Lemma \ref{tri-od:rot:twist:order:inv},  it immediately follows that the $p$ \emph{black} points $z_i$ of the \emph{branch} $b_{j+2}$ map to the last $p$ points of the \emph{branch} $b_j$ (in the direction away from $a$) in an \emph{order-preserving fashion}. In other words, $f(z_{p-i}) = x_{q-2p-i}$ for $i = 0, 1,2 \dots p-1$. 
	
	\item So, the remaining points $ x_1, x_2, \dots x_{q-3p}$ in the \emph{branch} $b_j$ must be images of the \emph{green points} $ x_p, x_{p+1}, \dots x_{q-2p}$ respectively in the same \emph{branch}. In other words, the last $ q-3p$ points in the \emph{branch} $b_j$ (in the direction away from  $a$) have the images \emph{shifted} $p$ points \emph{towards } $a$. 
\end{enumerate}

\begin{figure}[H]
	\caption{The \emph{unimodal triod twist patterns} $\Gamma_i^{ \frac{2}{7}}$ for $ i=1,2,3$ corresponding to \emph{rotation number} $\frac{2}{7} < \frac{1}{3}$}
	\centering
	\includegraphics[width=0.39\textwidth]{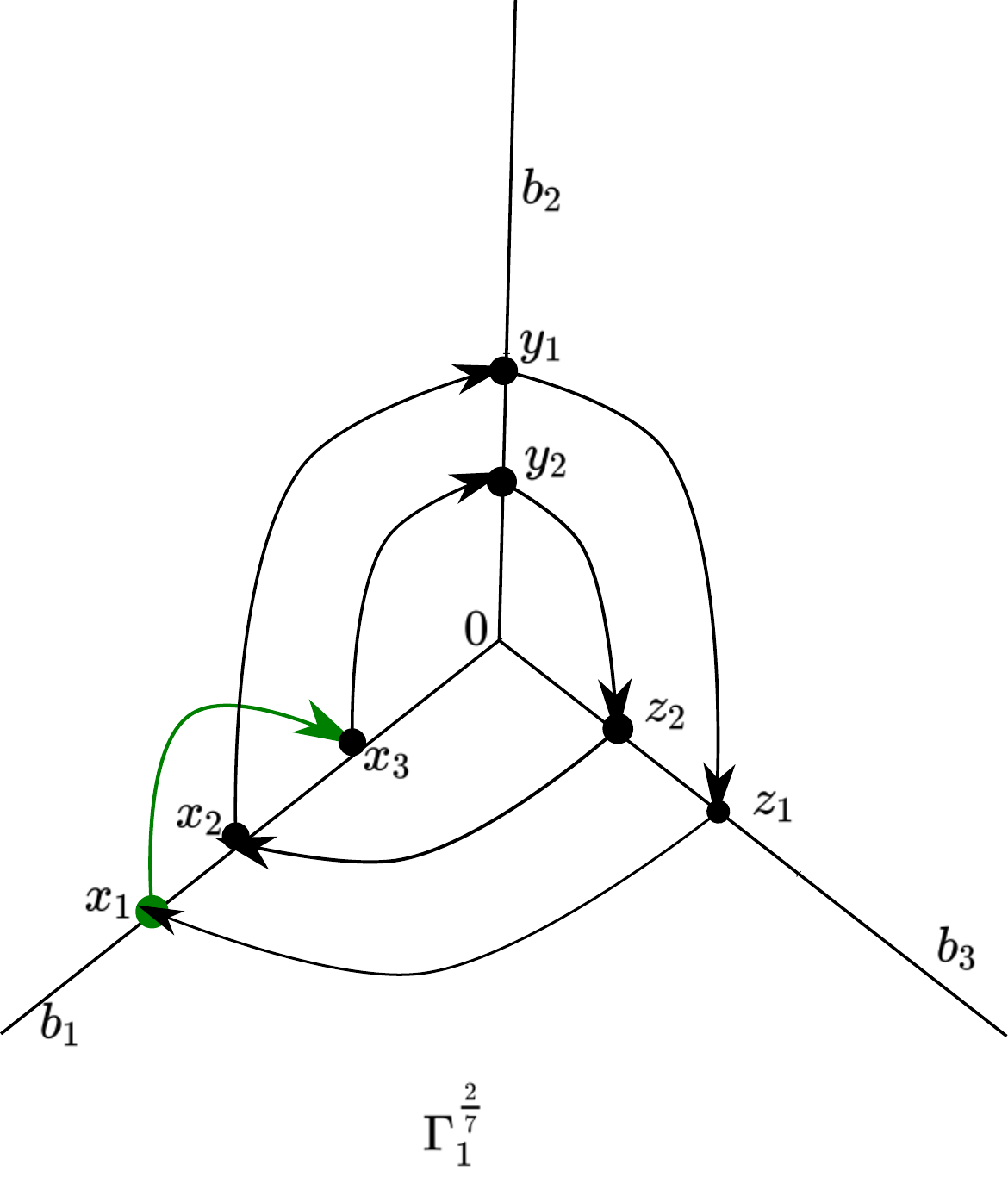}
	\includegraphics[width=0.39\textwidth]{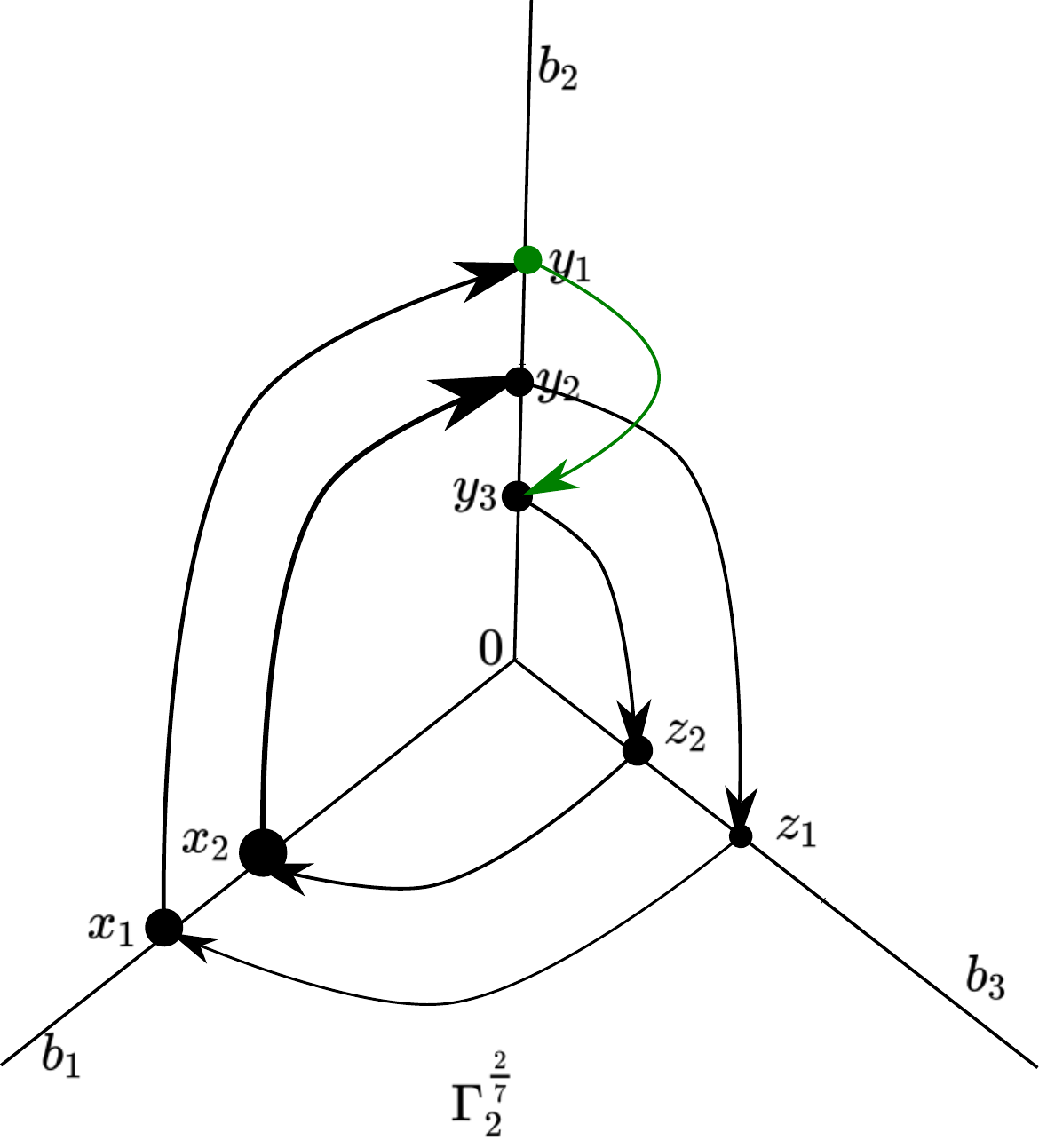}
	\includegraphics[width=0.39\textwidth]{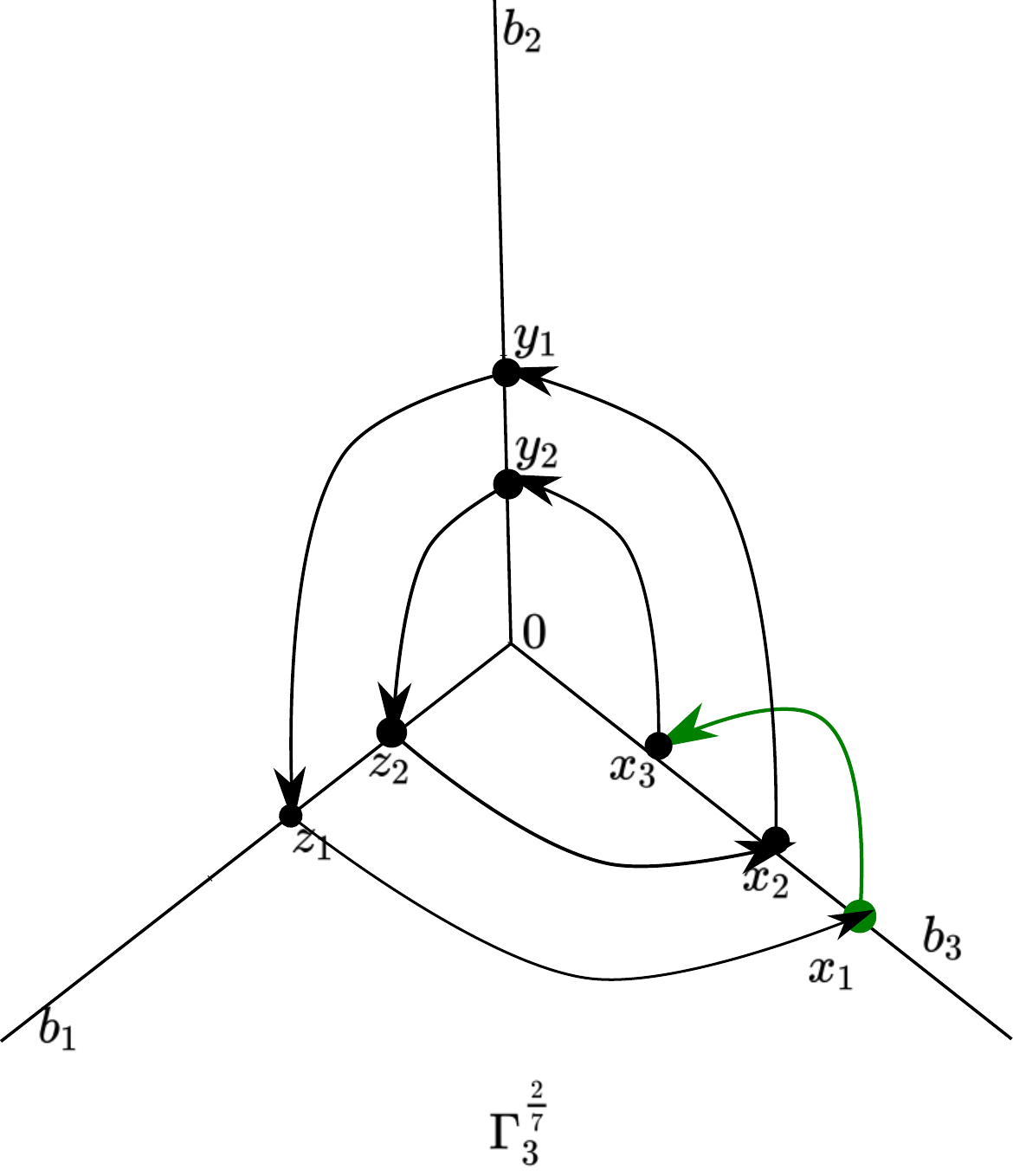}
	\label{drawing5}
\end{figure}

Now, we consider the case when $\frac{p}{q} > \frac{1}{3}$. In this case, by Lemma \ref{no:green},  $C$ is \emph{red}. If $ |C| = r$ and $ |B_j | = b$ then, computation gives us $ b = q-2p$ and $ r = 3p-q$.  Observe that if $ \frac{p}{q} \geqslant \frac{1}{2}$ then $b \leqslant 0$ which is not possible. So, no \emph{unimodal triod twist cycle} exists in this case. So, we set $ \frac{p}{q} < \frac{1}{2}$. We number the points of $P$ in the \emph{branches} $b_j$ , $b_{j+1}  $ and $b_{j+2}$ (in the direction away from  $a$) like before as $ x_1, x_2, \dots x_{p}$ ,  $ y_1, y_2, \dots y_{q-2p}$ and $z_1, z_2, \dots z_p$.

\begin{figure}[H]
	\caption{The \emph{unimodal triod twist patterns} $\Gamma_i^{ \frac{2}{5}}$ for $ i=1,2,3$ corresponding to \emph{rotation number} $\frac{2}{5} > \frac{1}{3}$}
	\centering
	\includegraphics[width=0.39\textwidth]{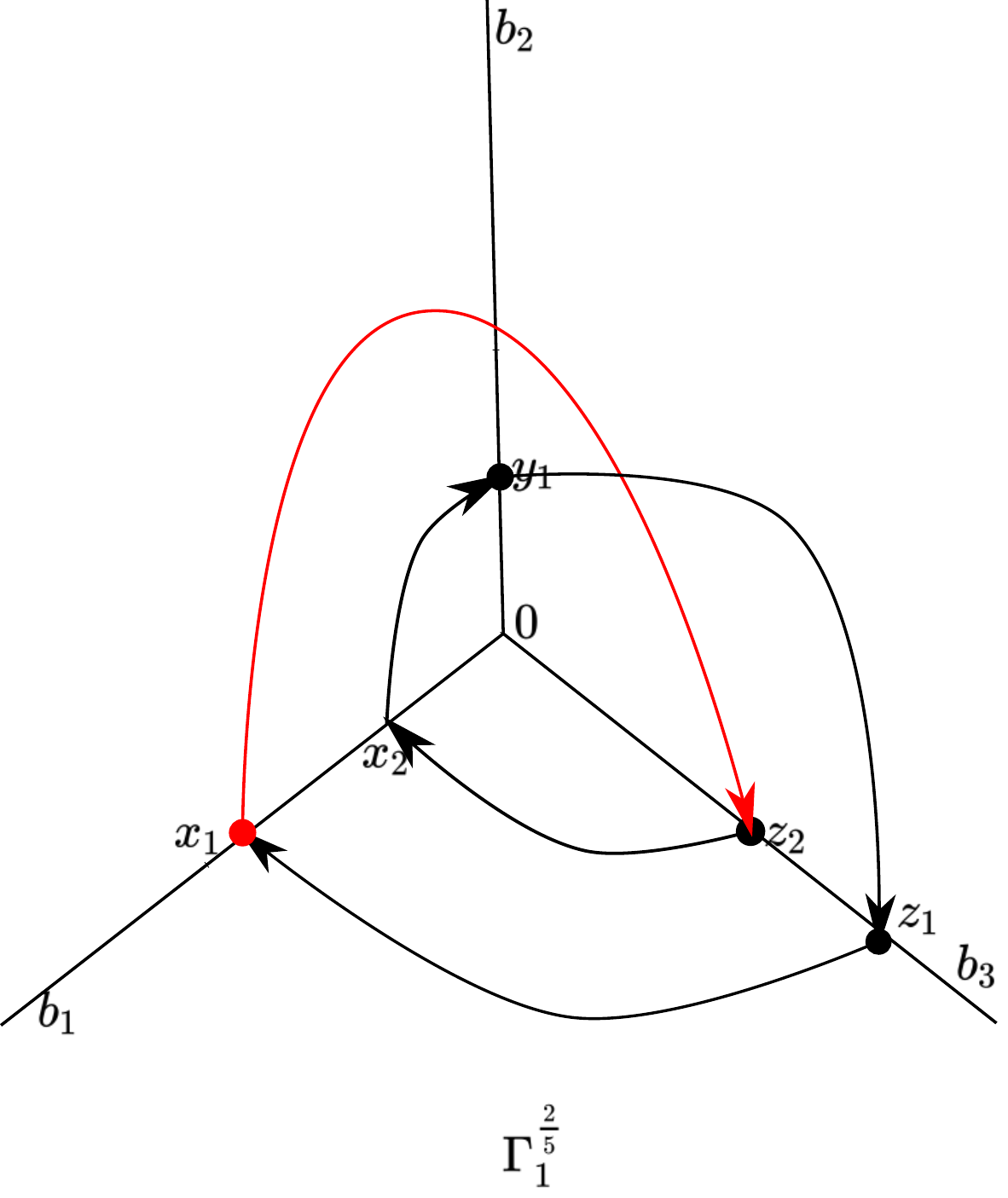}
	\includegraphics[width=0.39\textwidth]{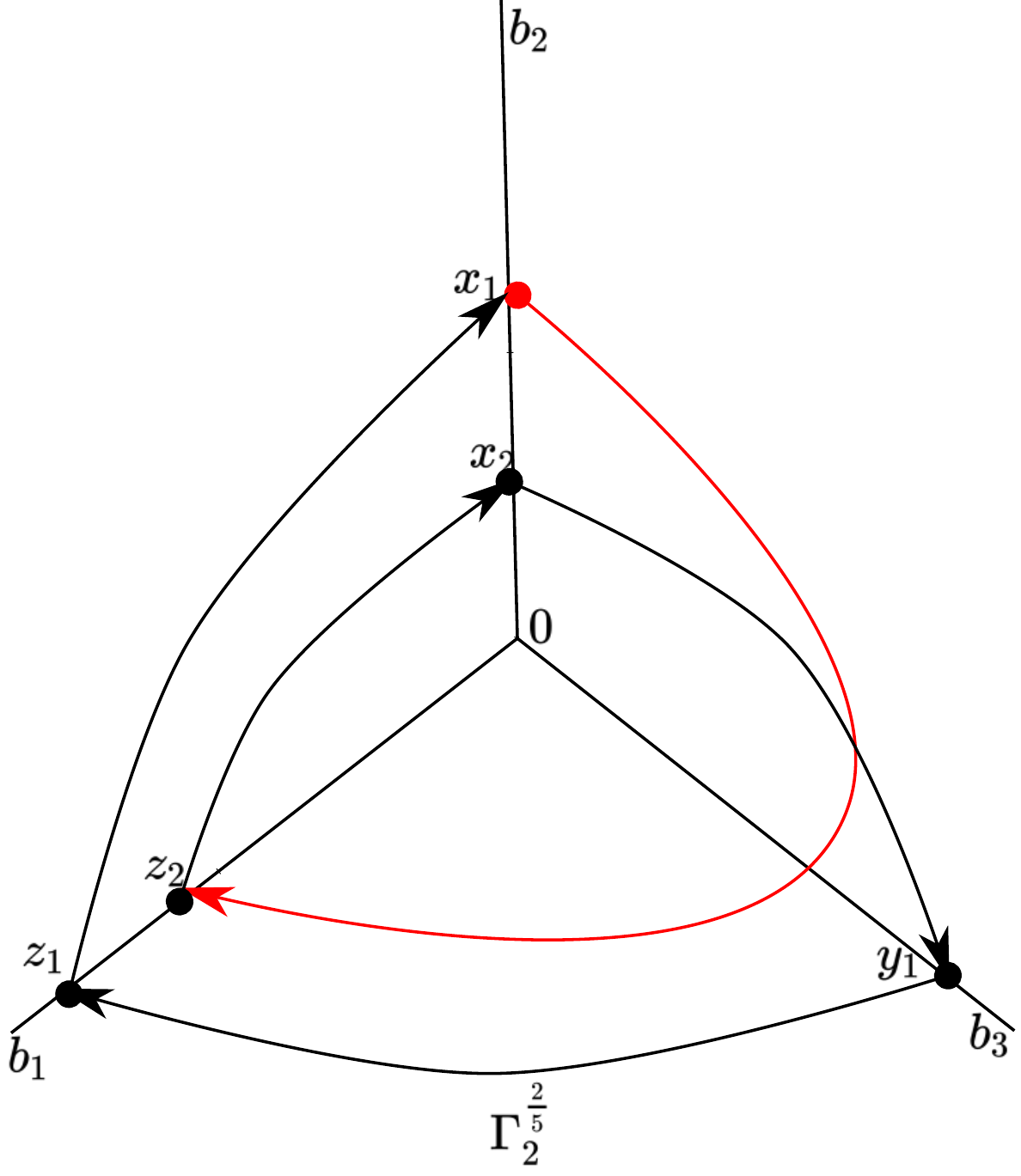}
	\includegraphics[width=0.39\textwidth]{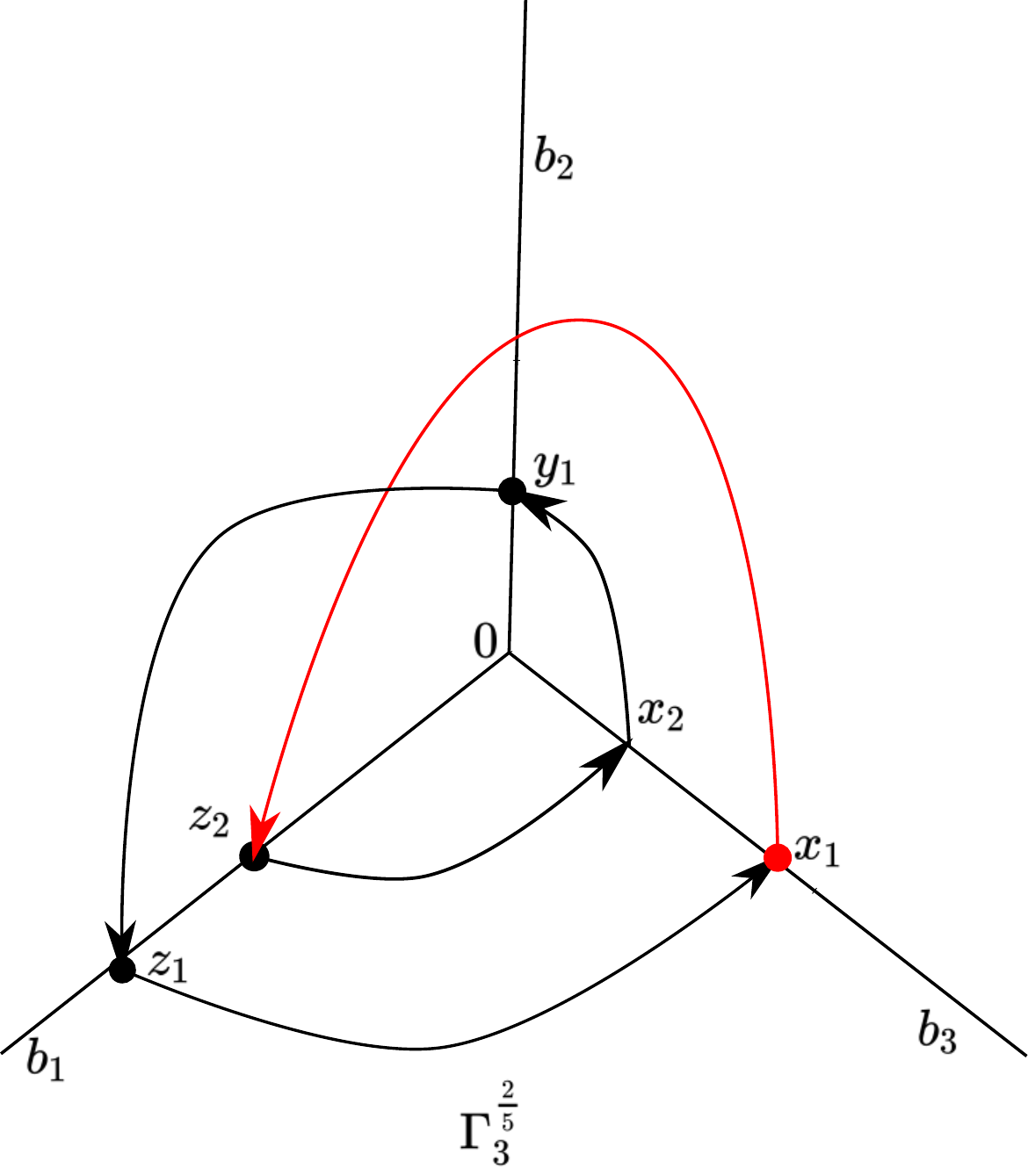}
	\label{drawing6}
\end{figure}

\begin{enumerate}
	\item By Lemma \ref{tri-od:rot:twist:order:inv}, it immediately follows that the first $ q-2p$  of the branch $b_j$ which are \emph{black} (in the direction away from  $a$) map to the $ q-2p$ points of the \emph{branch} $b_{j+1}$ in an \emph{order-preserving fashion}. In other words,  $f(x_i) = y_i$ for $ i = 1,2 \dots q-2p$.  Now, the  $p$ points of the \emph{branch} $b_{j+2}$ which are \emph{black} maps to $b_j$ in an \emph{order preserving fashion} . Since, $P$ is \emph{unimodal} , it follows that  $f(z_i) = x_i$ for $ i = 1,2 \dots p$. 
	
	\item To complete the orbit of $P$, since $P$ is unimodal, the last $ 3p-q$ points of the branch $b_j$ (in the direction away from  $a$)  which are \emph{red} must map to the first $ 3p-q$ points of the branch $b_{j+2}$ (in the direction away from  $a$),  that is, 
	 $f(x_{q-2p+i}) = z_i$ for $ i =1,2, \dots 3p-q$. 
	
	\item Thus, the $q-2p$ \emph{black  points} $y_1, y_2, \dots y_{q-2p}$  in the branch $b_{j+1}$ must map to the last $q-2p$ points  $ z_{3p-q+1},  z_{3p-q+2}, \dots z_p$  in the branch $b_{j+2}$ in an \emph{order preserving} fashion in the direction away from $a$. 
\end{enumerate}

 Observe that depending on the choice of the \emph{branch} $b_j, j=0,1,2$ containing the \emph{block} $C$, we get three distinct \emph{unimodal triod twist pattens} $\Gamma_j^{ \frac{p}{q} }$ , $ j=1,2,3$ with the same \emph{rotation number} $\frac{p}{q}$ (recall that we treat \emph{branches} to be distinguishable). 
 See Figure \ref{drawing5} for the case $\frac{p}{q} < \frac{1}{3} $ and Figure \ref{drawing6} for the case $\frac{p}{q} > \frac{1}{3} $. Lemma \ref{rot:one:third} together with the above leads us to the following Theorem.

\begin{theorem}\label{rho:less:more:third}
	For a given positive rational number $\frac{p}{q} $ strictly less than a half and not equal to a third, there exists three distinct unimodal triod twist patterns of rotation number $\frac{p}{q}$. There exists a unique unimodal triod twist pattern of rotation number one third and no unimodal triod twist pattern exists with rotation number greater than or equal to a half. 
\end{theorem}


\begin{thebibliography}{99999999}
		
	
\bibitem{alm98} L. Alsedà, J. Llibre, M. Misiurewicz, \emph{Periodic orbits of maps of Y} , Trans. Amer. Math. Soc. \textbf{313} (1989) 475–538.




\bibitem{akm65}	R. Adler, A. Konheim, and M. McAndrew,
\emph{Topological entropy}, Trans. Amer. Math. Soc. \textbf{114} (1965),
309--319.


\bibitem{alm00} Ll. Alsed\`{a}, J. Llibre and M. Misiurewicz,
\emph{Combinatorial Dynamics and Entropy in Dimension One},
Advanced Series in Nonlinear Dynamics (2nd edition) \textbf{5} (2000),
World Scientific Singapore (2000)

\bibitem{almnew}  Ll. Alsed\`{a}, J. Moreno, \emph{Linear orderings and the full periodicity kernel for the $n$-star}, Math.Anal. Appl. \textbf{180} (1993) 599-616. 

\bibitem{almm88} Ll. Alsed\`a, J. Llibre, F. Ma\~nosas and M.
Misiurewicz, \emph{Lower bounds of the topological entropy for
	continuous maps of the circle of degree one}, Nonlinearity
\textbf{1}(1988), 463--479.

\bibitem{ak79} J. Auslander, Y. Katznelson, \emph{Continuous maps
	of the circle without periodic points,} Israel J. of Math.
\textbf{32} (1979), 375--381.

\bibitem{Ba} S.Baldwin, \emph{Generalisation of a theorem of
	Sharkovsky on orbits of continous real valued functions,} Discrete
Math. \textbf{67} (1987), 111--127.

\bibitem{bb19} S. Bhattacharya, A. Blokh, \emph{Very badly ordered
	cycles of interval maps}, preprint arXiv:1908.06145 (2019)

\bibitem{bgmy80} L. Block, J. Guckenheimer, M. Misiurewicz and
L.-S. Young, \emph{Periodic points and topological entropy of
	one-dimensional maps}, Springer Lecture Notes in Mathematics
\textbf{819} (1980), 18--34.

\bibitem{blok84} A. Blokh, \emph{On transitive mappings of
	one-dimensional branched manifolds,} (Russian), In:
Differential-difference equations and problems of mathematical
physics (Russian), Akad. Nauk Ukrain. SSR, Inst. Mat., Kiev (1984),
3–-9.

\bibitem{blok86} \bysame, \emph{Dynamical systems on
	one-dimensional branched manifolds. I,} (Russian), Teor. Funktsii
Funktsional. Anal. i Prilozhen. \textbf{46}(1986), 8--18; translation
in J. Soviet Math. \textbf{48} (1990), no. 5, 500–-508.

\bibitem{blok87a} \bysame, \emph{Dynamical systems on
	one-dimensional branched manifolds. II,} (Russian), Teor. Funktsii
Funktsional. Anal. i Prilozhen. \textbf{47} (1987), 67--77; translation
in J. Soviet Math. \textbf{48} (1990), no. 6, 668–-674.


\bibitem{BB1} S. Bhattacharya, A. Blokh, \emph{Very badly ordered
	cycles of interval maps}, Journal of Difference Equations and
Applications \textbf{26} (2020), 1067-1084

\bibitem{BB2} S. Bhattacharya, A. Blokh, \emph{Over-rotation intervals of bimodal interval maps},
Journal of Difference Equations and Applications \textbf{26} (2020), 1085-1113

\bibitem{BB3} S. Bhattacharya, A. Blokh, \emph{Monotonicity of over-rotation intervals for bimodal interval maps},preprint arXiv:2103.03058 (2021)


\bibitem{blok87b} \bysame, \emph{Dynamical systems on
	one-dimensional branched manifolds. III,} (Russian), Teor. Funktsii
Funktsional. Anal. i Prilozhen. \textbf{48} (1987), 32--46; translation
in J. Soviet Math. \textbf{49} (1990), no. 2, 875-–883.

\bibitem{B1} \bysame, \emph{On Rotation Intervals for Interval
	Maps}, Nonlineraity \textbf{7}(1994), 1395--1417.

\bibitem{blo95a} \bysame, \emph{The Spectral Decomposition for
	One-Dimensional Maps}, Dynamics Reported \textbf{4} (1995), 1--59.

\bibitem{B2} \bysame, \emph{Rotation Numbers, Twists and a
	Sharkovsky-Misiurewicz-type Ordering for Patterns on the Interval},
Ergodic Theory and Dynamical Systems \textbf{15}(1995), 1--14. 		

\bibitem {B3} \bysame, \emph{Functional Rotation Numbers for
	One-Dimensional Maps,} Trans. Amer. Math. Soc. \textbf{347}(1995),
499--514

\bibitem{BM1} A. Blokh, M. Misiurewicz, \emph{A new order for
	periodic orbits of interval maps}, Ergodic Theory and Dynamical Sys.
\textbf{17}(1997), 565-574

\bibitem{BMR} A. Blokh, M. Misiurewicz, \emph{Rotation numbers for certain maps of an n-od}, Topology and its Applications
\textbf{114} (2001) 27–48

\bibitem{BM2} \bysame, \emph{Rotating an interval
	and a circle}, Trans. Amer. Math. Soc. \textbf{351}(1999), 63--78.

\bibitem{BS} A. Blokh, K. Snider, \emph{Over-rotation numbers for
	unimodal maps,} Journal of Difference Equations and Aplications
\textbf{19}(2013), 1108--1132.

\bibitem{Bo} J. Bobok, \emph{Twist systems on the interval}, Fund.
Math. \textbf{175}(2002), 97--117.

\bibitem{bk98} J. Bobok and M. Kuchta \emph{X-minimal orbits for
	maps on the interval}, Fund. Math. \textbf{156}(1998), 33--66.

\bibitem{cgt84} A. Chenciner, J.-M. Gambaudo and C. Tresser
\emph{Une remarque sur la structure des endo\-morphismes de degr\'e
	$1$ du cercle}, C. R. Acad. Sci. Paris, S\'er I Math. \textbf{299}
(1984), 145--148.

\bibitem{dgs76} M. Denker, C. Grillenberger,
K. Sigmund, \emph{Ergodic theory on compact spaces,} Lecture Notes
in Mathematics \textbf{527}(1976) Springer-Verlag, Berlin-New
York.

\bibitem{ito81} R. Ito, \emph{Rotation sets are closed}, Math.
Proc. Camb. Phil. Soc. \textbf{89}(1981), 107--111.

\bibitem{MT} J.Milnor and W.Thurston, \emph{On Iterated Maps on
	the Interval}, Lecture Notes in Mathematics, Springer, Berlin
\textbf{1342}(1988), 465--520.

\bibitem {mis82} M. Misiurewicz, \emph{Periodic points of maps
	of degree one of a circle,} Ergod. Th. \& Dynam. Sys. \textbf{2}(1982)
221--227.

\bibitem {mis89} \bysame, \emph{Formalism for studying
	periodic orbits of one dimensional maps}, European Conference
on Iteration Theory (ECIT 87), World Scientific
Singapore (1989), 1--7.	

\bibitem{MN} \bysame, \emph{Combinatorial Patterns for
	maps of the interval}, Mem. Amer. Math. Soc. \textbf{456}(1990) 		

\bibitem{mz89} M. Misiurewicz and K. Ziemian, \emph{Rotation Sets
	for Maps of Tori}, J. Lond. Math. Soc. (2) \textbf{40}(1989) 490--506.

\bibitem{npt83} S. Newhouse, J. Palis, F. Takens
\emph{Bifurcations and stability of families of diffeomorphisms},
Inst. Hautes \'Etudes Sci. Publ. Math. \textbf{57}(1983), 5--71.

\bibitem{poi} H. Poincar\'e, \emph{Sur les courbes d\'efinies par
	les \'equations diff\'erentielles}, Oeuvres completes, \textbf{1}
137--158, Gauthier-Villars, Paris (1952).

\bibitem{rt86} F. Rhodes, C. Thompson, \emph{Rotation numbers for
	monotone functions on the circle}, J. London Math. Soc. \textbf{34}(1986), 360--368.

\bibitem{S} A. N. Sharkovsky, \emph{Coexistence of the cycles of
	a continuous mappimg of the line into itself}, Ukraine Mat. Zh.
\textbf{16}(1964), 61--71 (Russian).

\bibitem{shatr} A. N. Sharkovsky, \emph{Coexistence of the
	cycles of a continuous mapping of the line into itself}, Internat.
J. Bifur. Chaos Appl. Sci. Engrg. \textbf{5}(1995), 1263--1273.

\bibitem {zie95}  K. Ziemian, \emph{ Rotation sets for subshifts
	of finite type,} Fundam. Math. \textbf{146}(1995), 189--201.
	

		
	\end{thebibliography}
\end{document}